\documentclass[12pt,oneside]{article}

\usepackage{amsmath}
\usepackage{amssymb}
\usepackage{color}
\usepackage{rotating}

\usepackage{graphicx}
\usepackage{subfigure}

\usepackage{url}

\author{Hui Zhou,
\\ \footnotesize{School of Mathematical Sciences, Peking University, Beijing, 100871, P.~R.~China}
\\ \footnotesize{zhouhpku17@pku.edu.cn, huizhou@math.pku.edu.cn, zhouhlzu06@126.com.},
\vspace{2em}
\\ Cheryl Praeger,~Michael Giudici,
\\ \footnotesize{School of Mathematics and Statistics, The University of Western Australia Crawley,}
\\ \footnotesize{Perth, WA 6009, Australia}
\\ \footnotesize{cheryl.praeger@uwa.edu.au, michael.giudici@uwa.edu.au}
\vspace{2em}
\\ Rongquan Feng~ and Xingui Fang
\\ \footnotesize{School of Mathematical Sciences, Peking University, Beijing, 100871, P.~R.~China}
\\ \footnotesize{fengrq@math.pku.edu.cn, xgfang@pku.edu.cn, xgfang@math.pku.edu.cn}  }

\title{On $s$-distance-transitive graphs}

\def\Aut{{\sf Aut}}

\def\S{{\sf S}}

\def\e{{\sf e}}
\def\G{{\sf G}}

\def\f{{\sf f}}
\def\Orb{{\sf Orb}}
\def\diam{{\rm diam}}
\def\deg{{\rm deg}}

\newtheorem{theorem}{Theorem}
\newtheorem{corollary}[theorem]{Corollary}
\newtheorem{lemma}[theorem]{Lemma}

\newtheorem{hypothesis}[theorem]{Hypothesis}
\newtheorem{question}[theorem]{Question}

\newtheorem{case}{Case}

\newcommand*{\QEDA}{\hfill\ensuremath{\blacksquare}}  
\newenvironment{proof}[1][\hspace{2ex}\textbf{\textit{Proof}.}\hspace{1ex}]{\begin{trivlist}\item[\hskip \labelsep {\bfseries #1}]}{\QEDA\end{trivlist}}

\begin{document}

\maketitle

\begin{abstract}



Distance-regular graphs have many beautiful combinatorial properties. Distance-transitive graphs have very strong symmetries, and they are distance-regular, i.e., distance-transitivity implies distance-regularity. In this paper, we give similar results, i.e., for special $s$ and graphs with other restrictions we show that $s$-distance-transitivity implies distance-regularity.

\end{abstract}




\section{Introduction}\label{section Introduction}


In this paper, all graphs are finite, simple, connected and undirected. An ordered pair of adjacent vertices is called an arc. Let $\Gamma$ be a graph. We use $V\Gamma$, $E\Gamma$, $A\Gamma$ and $\Aut(\Gamma)$ to denote the vertex-set, the edge-set, the arc-set and the full automorphism group of $\Gamma$, respectively. The distance between two vertices $u$ and $v$ of $\Gamma$, denoted by $\partial_\Gamma(u,v)$, is the length of a shortest path connecting $u$ and $v$ in $\Gamma$. The diameter of $\Gamma$, denoted by $\diam(\Gamma)$, is the maximum distance occurring over all pairs of vertices. Fix a vertex $v\in V\Gamma$. For $0\leqslant i\leqslant \diam(\Gamma)$, we use $\Gamma_i(u)$ to denote the set of vertex $u$ with $\partial_\Gamma(u,v)=i$. For convenience, we usually use $\Gamma(v)$ to denote $\Gamma_1(v)$. The degree of $v$, denoted by $\deg_\Gamma(v)$ or simply $\deg(v)$, is the number of vertices adjacent to $v$ in $\Gamma$, i.e. $\deg_\Gamma(v)=|\Gamma(v)|$. 
The graph $\Gamma$ is called regular with valency $k$ (or $k$-regular) if the degree of each vertex of $\Gamma$ is $k$. The girth of $\Gamma$ is the length of a shortest cycle of $\Gamma$.

Let $G\leqslant \Aut(\Gamma)$. The graph $\Gamma$ is called $G$-vertex-transitive ($G$-arc-transitive, respectively), if $G$ is transitive on the vertex-set $V\Gamma$ (the arc-set $A\Gamma$, respectively). Let $s\geqslant 1$. An $s$-arc of $\Gamma$ is an $(s+1)$-tuple of vertices of $\Gamma$ in which every two consecutive vertices are adjacent and every three consecutive vertices are pairwise distinct. The graph $\Gamma$ is called $(G,s)$-arc-transitive if it is $G$-vertex-transitive and $G$ is also transitive on $s$-arcs of $\Gamma$. The graph $\Gamma$ is called $(G,s)$-distance-transitive, if for each $1\leqslant i\leqslant s$, the group $G$ is transitive on the orderer pairs of form $(u,v)$ with $\partial_\Gamma(u,v)=i$. A $(G,s)$-distance-transitive graph is called $G$-distance-transitive (or distance-transitive, respectively), if $s=\diam(\Gamma)$ (and $G=\Aut(\Gamma)$, respectively). By definitions, $s$-arc-transitivity implies $s$-distance-transitivity.

The graph $\Gamma$ is called distance-regular if it is $k$-regular with diameter $d$ and there exists natural numbers $b_0=k$, $b_i$~($1\leqslant i\leqslant d-1$), $c_1=1$ and $c_j$~($2\leqslant j\leqslant d$) such that for each pair $(u,v)$ of vertices with $\partial_\Gamma(u,v)=h$ we have $|\Gamma_{h-1}(u)\cap\Gamma(v)|=c_h$~($1\leqslant h\leqslant d$) and $|\Gamma_{h+1}(u)\cap\Gamma(v)|=b_h$~($0\leqslant h\leqslant d-1$). The array $\iota(\Gamma)=\{k,b_1,\ldots,b_{d-1};1,c_2,\ldots,c_d\}$ is the intersection array of $\Gamma$. Let $a_0=0$ and $a_d=k-c_d$. For $1\leqslant h\leqslant d-1$, let $a_h=k-c_h-b_h$. Then the intersection array of $\Gamma$ can also be written as
\begin{equation*}
\iota(\Gamma)=\left\{\begin{array}{cccccc}
               *     & c_1=1 & c_2 & \cdots & c_{d-1} & c_d\\
               a_0=0 & a_1 & a_2 & \cdots & a_{d-1} & a_d\\
               b_0=k & b_1 & b_2 & \cdots & b_{d-1} & *
               \end{array}\right\}.
\end{equation*}

\newpage

By definitions, distance-transitivity implies distance-regularity. Distance-transitive graphs of small valencies were classified. Distance-transitive graphs of valency three were classified by Biggs and Smith~\cite{Biggs and Smith On trivalent graphs}, distance-transitive graphs of valency four were classified by Smith~\cite{Smith Distance-transitive graphs of valency four,Smith On bipartite tetravalent graphs,Smith On tetravalent graphs}, etc.~\cite[Section~7.5]{BCN}. These graphs are distance-regular. Similarly, distance-regular graphs of small valencies were classified. For example, cubic distance-regular graphs were classified by Biggs, Boshier and Shawe-Taylor~\cite{Cubic distance-regular graphs}, tetravalent distance-regular graphs were classified by Brouwer and Koolen~\cite{4dr}.


We consider for which $s$ not greater than the diameter of the graph we also have the property that $s$-distance-transitivity implies distance-regularity. The symmetry of $s$-distance-transitive graphs with $s$ equal to the diameter minus one is very close to distance-transitive graphs. In this paper, we get the following main result.

\begin{theorem}[Main Theorem]\label{thm main theorem}
Let $\Gamma$ be a $k$-regular graph with diameter $d\geqslant 2$ and girth $g$ where $k\geqslant 3$. Let $s=d-1$, and let $G\leqslant \Aut(\Gamma)$ such that $\Gamma$ is $(G,s)$-distance-transitive.
\begin{enumerate}
\item Let $k=3$. Then $\Gamma$ is distance-regular. Furthermore, let $d\geqslant 3$, if $\Gamma$ is not $G$-distance-transitive, then $g\leqslant 2d-1$.
\item Let $k=4$ and $d\geqslant 3$. If $g\geqslant 2d$, then $\Gamma$ is distance-regular. Furthermore, if $\Gamma$ is not $G$-distance-transitive, then $g\leqslant 2d-1$.
\end{enumerate}
\end{theorem}

To prove the main theorem, we investigate the adjacency relations between $\Gamma_s(\alpha)$ and an orbit $\Delta$ in $\Gamma_{s+1}(\alpha)$ for a graph $\Gamma$ with large girth $g\geqslant 2s+2$ where $\alpha\in V\Gamma$ and $s\geqslant 1$. In Section~\ref{sec Local actions and designs}, we show that these adjacency relations correspond to designs. In Section~\ref{sec Adjacency relations}, we give some special adjacency relations. In Section~\ref{sec cubic s-Distance-transitive graphs}, we show the main theorem holds for cubic graphs. In Section~\ref{sec tetravalent s-Distance-transitive graphs}, we show that the main theorem holds for tetravalent graphs.

By our main theorem, one may consider the relation between the girth and the diameter of such $s$-distance-transitive graphs. The following question comes out.

%

\begin{question}\label{question d-1 dt girth <2d}
Let $k\geqslant 3$ and let $\Gamma$ be a $k$-regular graph with diameter $d\geqslant 3$ and girth $g$. Let $s=d-1$ and let $G\leqslant \Aut(\Gamma)$ such that $\Gamma$ is $(G,s)$-distance-transitive but not $G$-distance-transitive. Does the inequality $g\leqslant 2d-1$ hold? 
\end{question}

Our main theorem answers Question~\ref{question d-1 dt girth <2d} for cubic and tetravalent graphs. The case of valency $k\geqslant 5$ of Question~\ref{question d-1 dt girth <2d} is still open.

%
%
%



\section{Local girth and intersection numbers}\label{sec Local girth and intersection numbers}

Let $\Gamma$ be a graph and let $\alpha\in V\Gamma$. The $\alpha$-girth of $\Gamma$ denoted by $\lambda_\alpha(\Gamma)$, is the length of the shortest cycle containing $\alpha$ in $\Gamma$. The local girth of $\Gamma$ at $\alpha$ is defined to be $\lambda(\Gamma,\alpha)=\min\{\lambda_\alpha(\Gamma),\lambda_\beta(\Gamma)+1\mid \beta\in \Gamma_1(\alpha)\}$. The following lemma is straightforward.

\begin{lemma}\label{lem alpha girth >=2s+2}
Let $\Gamma$ be a graph with valency $k\geqslant 3$. Let $\alpha\in V\Gamma$ and suppose the $\alpha$-girth $\lambda_\alpha(\Gamma)\geqslant 2s+2$ where $1\leqslant s <\varepsilon_\Gamma(\alpha)$. Then for any $1\leqslant i\leqslant s$ and for any $\beta\in \Gamma_i(\alpha)$, we have
\begin{enumerate}
\item there is a unique shortest path of length $i$ connecting $\alpha$ and $\beta$,
\item the intersection number $c_i(\Gamma,\alpha)=|\Gamma_1(\beta)\cap \Gamma_{i-1}(\alpha)|=1$,
\item the intersection number $a_i(\Gamma,\alpha)=|\Gamma_1(\beta)\cap \Gamma_i(\alpha)|=0$,
\item the intersection number $b_i(\Gamma,\alpha)=|\Gamma_1(\beta)\cap \Gamma_{i+1}(\alpha)|=k-1$, and
\item the $s$-partial intersection array of $\Gamma$ at $\alpha$ is
\begin{eqnarray}\label{eqn s-partial intersection array at alpha}
\iota(\Gamma,\alpha,s)&=&\left\{\begin{array}{ccccc}
                       * & c_1(\Gamma,\alpha) & c_2(\Gamma,\alpha) & \cdots & c_s(\Gamma,\alpha)\\
                       a_0(\Gamma,\alpha)=0 & a_1(\Gamma,\alpha) & a_2(\Gamma,\alpha) & \cdots & a_s(\Gamma,\alpha)\\
                       b_0(\Gamma,\alpha)=k & b_1(\Gamma,\alpha) & b_2(\Gamma,\alpha) & \cdots & b_s(\Gamma,\alpha)
                       \end{array}\right\}\nonumber\\
               &=&\left\{\begin{array}{ccccc}
                       * & 1 & 1 & \cdots & 1\\
                       0 & 0 & 0 & \cdots & 0\\
                       k & k-1 & k-1 & \cdots & k-1
                       \end{array}\right\}.\nonumber
\end{eqnarray}
\end{enumerate}
\end{lemma}
%
%

\begin{figure}[htb]
\centering
\includegraphics[width=0.40\textwidth]{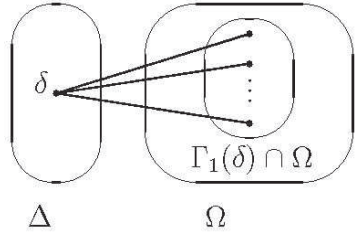}
\caption{The intersection number $\kappa(\Gamma,G;\Delta,\Omega)$ from $\Delta$ to $\Omega$.}
\label{fig: intersectionnumbersbetweenorbits}
\end{figure}

%
%
%
%

\begin{lemma}\label{lem intersection number between orbits}
Let $\Gamma$ be a graph and let $G\leqslant \Aut(\Gamma)$. Let $\Omega$ and $\Delta$ be two subsets of $V\Gamma$ such that $G$ is transitive on $\Delta$ and every element of $G$ fixes $\Omega$ setwise. Let $\delta\in \Delta$. Then the intersection number \begin{center}$\kappa(\Gamma,G;\Delta,\Omega)=|\Gamma_1(\delta)\cap \Omega|$\end{center} is independent of the choice of $\delta$. If there have edges between $\Delta$ and $\Omega$, then $\kappa(\Gamma,G;\Delta,\Omega)\geqslant 1$.
\end{lemma}

\begin{proof}
Take any $\alpha\in \Delta$. Then there exists $g\in G$ such that $\alpha=\delta^g$. We have $(\Gamma_1(\delta)\cap \Omega)^g\subseteq \Gamma_1^g(\delta)\cap \Omega^g=\Gamma_1(\delta^g)\cap \Omega=\Gamma_1(\alpha)\cap \Omega$. By the same argument, we have $(\Gamma_1(\alpha)\cap \Omega)^{g^{-1}}\subseteq \Gamma_1(\delta)\cap \Omega$ and so $\Gamma_1(\alpha)\cap \Omega=(\Gamma_1(\alpha)\cap \Omega)^{g^{-1}g}\subseteq (\Gamma_1(\delta)\cap \Omega)^g$. Then $(\Gamma_1(\delta)\cap \Omega)^g=\Gamma_1(\alpha)\cap \Omega$. Hence $g$ induces a one-to-one map from $\Gamma_1(\delta)\cap \Omega$ to $\Gamma_1(\alpha)\cap \Omega$, which implies $|\Gamma_1(\delta)\cap \Omega|=|\Gamma_1(\alpha)\cap \Omega|$. This means the intersection number $|\Gamma_1(\delta)\cap \Omega|$ is independent of the choice of $\delta$.
\end{proof}

Let $G$ be a group acting on a set $\Omega$. We use $\Orb(G;\Omega)$ to denote the set of $G$-orbits in $\Omega$. The following corollary is a direct consequence of Lemma~\ref{lem alpha girth >=2s+2} and Lemma~\ref{lem intersection number between orbits}.

\begin{figure}[htb]
\centering
\includegraphics[width=0.60\textwidth]{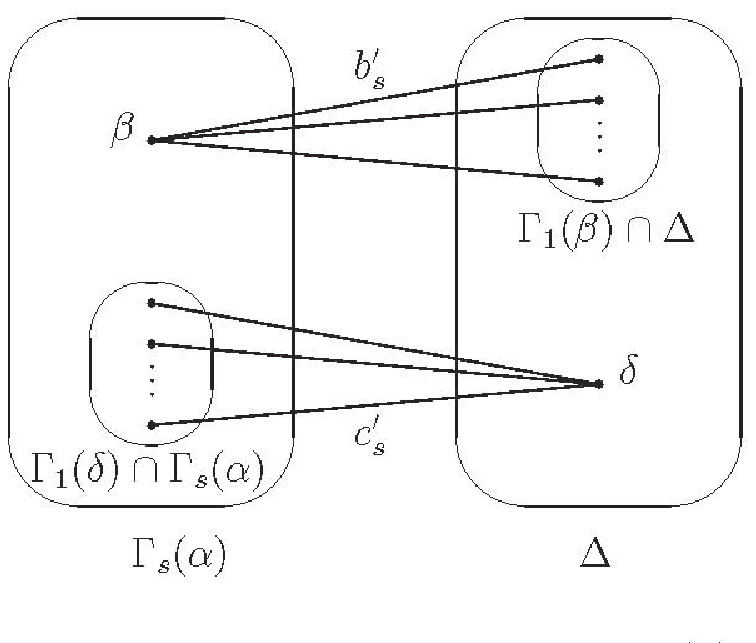}
\caption{The intersection number between $\Gamma_s(\alpha)$ and $\Delta$.}
\label{fig: intersectionnumbersbetweenorbitsR}
\end{figure}

%
%
%

\begin{corollary}\label{cor s+1 orbit}
Let $\Gamma$ be a graph with valency $k\geqslant 3$ and let $G\leqslant \Aut(\Gamma)$. Let $\alpha\in V\Gamma$ and suppose the $\alpha$-girth $\lambda_\alpha(\Gamma)\geqslant 2s+2$ where $s\geqslant 1$. Suppose $G_\alpha$ is transitive on $\Gamma_s(\alpha)$. Let $\Delta$ be a $G_\alpha$-orbit in $\Gamma_{s+1}(\alpha)$. Let $\beta\in \Gamma_s(\alpha)$ and $\delta\in \Delta$. Then
\begin{enumerate}
\item $G_\alpha$ is transitive on $\Gamma_i(\alpha)$ for $1\leqslant i\leqslant s$,
\item $b_s'(\Gamma,G_\alpha,\Delta)=|\Gamma_1(\beta)\cap \Delta|\geqslant 1$ is independent of the choice of $\beta$,
\item $c_s'(\Gamma,G_\alpha,\Delta)=|\Gamma_1(\delta)\cap \Gamma_s(\alpha)|\geqslant 1$ is independent of the choice of $\delta$, and
\item $b_s(\Gamma,\alpha)=\sum\limits_{\Delta\in \Orb(G_\alpha;\Gamma_{s+1}(\alpha))}b_s'(\Gamma,G_\alpha;\Delta)=k-1$.
\end{enumerate}
\end{corollary}

Note that,
\begin{eqnarray*}
c_s'(\Gamma,G_\alpha,\Delta)&=&\kappa(\Gamma,G_\alpha,\Delta,\Gamma_s(\alpha)),\\
b_s'(\Gamma,G_\alpha,\Delta)&=&\kappa(\Gamma,G_\alpha,\Gamma_s(\alpha),\Delta)
\end{eqnarray*}
in the above corollary. If the graph in Corollary~\ref{cor s+1 orbit} is vertex-transitive, then the intersection numbers are independent of the vertex $\alpha$ and they have the following properties.


\begin{lemma}\label{lem s-dt properties of intersection numbers}
Let $\Gamma$ be a $(G,s)$-distance-transitive graph with valency $k\geqslant 3$ where $G\leqslant \Aut(\Gamma)$. Let $1\leqslant i\leqslant s$. Then the intersection numbers
\begin{eqnarray*}
c_i(\Gamma)&=&c_i(\Gamma,\alpha)=\kappa(\Gamma,G_\alpha;\Gamma_i(\alpha),\Gamma_{i-1}(\alpha)), \\ a_i(\Gamma)&=&a_i(\Gamma,\alpha)=\kappa(\Gamma,G_\alpha;\Gamma_i(\alpha),\Gamma_i(\alpha)), \\ b_i(\Gamma)&=&b_i(\Gamma,\alpha)=\kappa(\Gamma,G_\alpha;\Gamma_i(\alpha),\Gamma_{i+1}(\alpha)),\\
c_s'(\Gamma)&=&c_s'(\Gamma,\alpha)=\min\{c_s'(\Gamma,G_\alpha,\Delta)\mid \Delta\in \Orb(G_\alpha;\Gamma_{s+1}(\alpha))\}
\end{eqnarray*}
are independent of $\alpha$. Furthermore, they have the following relations
\begin{eqnarray*}
1&=&c_1(\Gamma)\leqslant c_2(\Gamma)\leqslant c_3(\Gamma)\leqslant \cdots\leqslant c_s(\Gamma)\leqslant c_s'(\Gamma),\\
k&=&b_0(\Gamma)>b_1(\Gamma)\geqslant b_2(\Gamma)\geqslant\cdots\geqslant b_s(\Gamma).
\end{eqnarray*}
\end{lemma}

\begin{proof}
Since $\Gamma$ is $G$-vertex-transitive, it is obvious that the intersection numbers are independent of $\alpha$.

Now we give the proof of the relations of $c_i(\Gamma)$'s. Take a vertex $\alpha\in V\Gamma$ and an arbitrary $G_\alpha$-orbit $\Delta$ in $\Gamma_{s+1}(\alpha)$. We first show that $c_s(\Gamma,\alpha)\leqslant c_s'(\Gamma,G_\alpha,\Delta)$. Take a vertex $\delta\in \Delta$. The distance $\partial_\Gamma(\alpha,\delta)=s+1$. Let $\beta\in \Gamma(\alpha)$ such that $\beta$ lies on a shortest path connecting $\alpha$ and $\delta$. Then $\partial_\Gamma(\beta,\delta)=s$. Assume $w\in \Gamma_{s-1}(\beta)\cap\Gamma(\delta)$. Then $\partial_\Gamma(\alpha,w)=s$, i.e. $w\in \Gamma_s(\alpha)\cap\Gamma(\delta)$. This implies $\Gamma_{s-1}(\beta)\cap\Gamma(\delta)\subseteq \Gamma_s(\alpha)\cap\Gamma(\delta)$. Hence $c_s(\Gamma,\beta)=|\Gamma_{s-1}(\beta)\cap\Gamma(\delta)|\leqslant |\Gamma_s(\alpha)\cap\Gamma(\delta)|=c_s'(\Gamma,G_\alpha,\Delta)$. Since $\Delta$ is arbitrary, we thus have $c_s(\Gamma)=c_s(\Gamma,\beta)\leqslant c_s'(\Gamma,\alpha)=c_s'(\Gamma)$. The proof of $c_i(\Gamma)\leqslant c_{i+1}(\Gamma)$, where $1\leqslant i\leqslant s-1$, are similar as above.

Now we give the proof of $b_i(\Gamma)\geqslant b_{i+1}(\Gamma)$ where $1\leqslant i\leqslant s-1$. Take a vertex $\alpha\in V\Gamma$ and a vertex $\delta\in \Gamma_{i+1}(\alpha)$. Let $\beta\in \Gamma(\alpha)$ such that $\beta$ lies on a shortest path connecting $\alpha$ and $\delta$. Then $\partial_\Gamma(\beta,\delta)=i$. Assume $w\in \Gamma_{i+2}(\alpha)\cap\Gamma(\delta)$. Then $\partial_\Gamma(\beta,w)=i+1$, i.e. $w\in \Gamma_{i+1}(\beta)\cap\Gamma(\delta)$. This implies $\Gamma_{i+2}(\alpha)\cap\Gamma(\delta)\subseteq \Gamma_{i+1}(\beta)\cap\Gamma(\delta)$. Hence $b_{i+1}(\Gamma,\alpha)=|\Gamma_{i+2}(\alpha)\cap\Gamma(\delta)|\leqslant |\Gamma_{i+1}(\beta)\cap\Gamma(\delta)|=b_i(\Gamma,\beta)$, i.e. $b_{i+1}(\Gamma)\leqslant b_i(\Gamma)$.
\end{proof}



\section{Local actions and designs}\label{sec Local actions and designs}

We recall the definition and properties of designs~\cite{Combinatorial Designs Constructions and Analysis}. Let $t,c,k$ be integers satisfying $1 \leqslant t \leqslant c \leqslant k$, and let $\lambda$ be a positive integer. Let $X$ be a finite set with cardinality $|X|=k$. Let $\mathcal{B}$ be a subset of $\{S\subseteq X\mid |S|=c\}$. The pair $(X,\mathcal{B})$ is called a $t$-$(k,c,\lambda)$ design if for any $t$-subset $T$ of $X$, the size $|\{B\in \mathcal{B}\mid T\subseteq B\}|=\lambda$ is independent of $T$. Elements in $\mathcal{B}$ are called blocks. If $(X,\mathcal{B})$ is a $t$-$(k,c,\lambda)$ design ($t$-design for short), then it becomes a $j$-design for each $1\leqslant j\leqslant t$, since $|\{B\in \mathcal{B}\mid I\subseteq B\}|=\lambda_j$ is independent of the choice of the $j$-subset $I$ of $X$. Note that $\lambda_j=\lambda\binom{k-j}{t-j}\bigl/\binom{c-j}{t-j}$ for each $0\leqslant j\leqslant t$, $\lambda=\lambda_t$, the number of blocks is $b=|\mathcal{B}|=\lambda_0=\lambda\binom{k}{t}\bigl/\binom{c}{t}$ and each element $x\in X$ occurs in $\lambda_1=\lambda\binom{k-1}{t-1}\bigl/\binom{c-1}{t-1}$ blocks.

\begin{hypothesis}\label{hypothesis}
Let $\Gamma$ be a graph with valency $k\geqslant 3$ and let $G\leqslant \Aut(\Gamma)$. Let $\alpha\in V\Gamma$ and suppose the local girth $\lambda(\Gamma,\alpha)\geqslant 2s+2$ where $s\geqslant 1$. Suppose $G_\alpha$ is transitive on $\Gamma_s(\alpha)$. Let $\Delta$ be a $G_\alpha$-orbit in $\Gamma_{s+1}(\alpha)$. Suppose the local action of $G_\alpha$ on $\Gamma_1(\alpha)$ is $t$-homogeneous for some $1\leqslant t\leqslant c_s'$ where $c_s'=c_s'(\Gamma,G_\alpha,\Delta)\leqslant k$.
\end{hypothesis}

%

From now on, in this section we suppose Hypothesis~\ref{hypothesis} holds. By Lemma~\ref{lem alpha girth >=2s+2}, for $1\leqslant i\leqslant s$, the intersection numbers
\begin{center}$(c_i(\Gamma,\alpha),a_i(\Gamma,\alpha),b_i(\Gamma,\alpha))=(1,0,k-1)$,\end{center} there is no edge between vertices in $\Gamma_i(\alpha)$ and \begin{center}$k_i=|\Gamma_i(\alpha)|=k(k-1)^{i-1}$.\end{center} Let $\Gamma_1(\alpha)=\{\beta_1,\beta_2,\ldots,\beta_k\}$. For each $1\leqslant j\leqslant k$, let
\begin{eqnarray*}
B_j&=&\Gamma_s(\alpha)\cap \Gamma_{s-1}(\beta_j),\\
C_j&=&\bigcup\limits_{\gamma\in B_j}\Gamma_1(\gamma)\cap\Gamma_{s+1}(\alpha)=\Gamma_{s+1}(\alpha)\cap\Gamma_s(\beta_j).
\end{eqnarray*}
Then \begin{center}$|B_j|=(k-1)^{s-1}$\end{center} and $|C_j|\leqslant b_s(\Gamma,\alpha)|B_j|=(k-1)|B_j|$. Since $\lambda_{\beta_j}(\Gamma)+1\geqslant \lambda(\Gamma,\alpha)\geqslant 2s+2$, we have $\lambda_{\beta_j}(\Gamma)\geqslant 2s+1$, and so \begin{center}$|C_j|= (k-1)|B_j|=(k-1)^s$.\end{center} The sets $B_1,B_2,\ldots,B_k$ form a partition of $\Gamma_s(\alpha)$. Let
\begin{equation}\label{equation X=[k]}
X=\{1,2,\ldots,k\}.
\end{equation}

Let $g\in G_\alpha$ such that $g$ maps $\beta_j\in\Gamma_1(\alpha)$ to $\beta_h\in\Gamma_1(\alpha)$. Then $C_j^g=\bigl(\Gamma_{s+1}(\alpha)\cap \Gamma_s(\beta_j)\bigr)^g\subseteq \Gamma_{s+1}^g(\alpha)\cap \Gamma_s^g(\beta_j)=\Gamma_{s+1}(\alpha)\cap \Gamma_s(\beta_j^g)=\Gamma_{s+1}(\alpha)\cap \Gamma_s(\beta_h)=C_h$, that is $C_j^g\subseteq C_h$. By the same argument, we have $C_h^{g^{-1}}\subseteq C_j$, and so $C_h=C_h^{g^{-1}g}\subseteq C_j^g$. Hence we have $C_j^g=C_h$. This means the vertex stabilizer $G_\alpha$ is transitive on the set $\{C_j \mid j\in X\}$. Similarly, since $G_\alpha$ is $t$-homogeneous on $\Gamma_1(\alpha)$, we have that $G_\alpha$ is transitive on the set $\Biggl\{\bigcap\limits_{j\in A}C_j~\Biggl|~ A\subseteq X \text{ with } |A|=a\Biggr\}$ where $1\leqslant a\leqslant t$.


The intersection numbers between $\Gamma_s(\alpha)$ and $\Delta$ are $b_s'$ and $c_s'$ where $b_s'=b_s'(\Gamma,G_\alpha,\Delta)\geqslant 1$ and $1\leqslant c_s'=c_s'(\Gamma,G_\alpha,\Delta)\leqslant k$. For $1\leqslant j\leqslant k$, let
\begin{eqnarray*}
\Delta_j&=&\bigcup\limits_{\gamma\in B_j}\Delta\cap\Gamma_1(\gamma)=\Delta\cap\Gamma_s(\beta_j)=\Delta\cap C_j,\\
\Delta_j'&=&\Delta\setminus \Delta_j.
\end{eqnarray*}
The third equation in the line of $\Delta_j$ holds as $\Delta\cap\Gamma_s(\beta_j)=\bigl(\Delta\cap\Gamma_{s+1}(\alpha)\bigr)\cap\Gamma_s(\beta_j)=\Delta\cap\bigl(\Gamma_{s+1}(\alpha)\cap\Gamma_s(\beta_j)\bigr)=\Delta\cap C_j$. Then $|\Delta_j|\leqslant b_s'|B_j|$. Since $\lambda_{\beta_j}(\Gamma)\geqslant 2s+1$, we have that for each $\delta\in \Delta$ the size
\begin{equation}\label{eqn Gamma1delta cap Bj <=1}
|\Gamma_1(\delta)\cap B_j|\leqslant 1,
\end{equation}
and \begin{center}$\Bigl(\Delta\cap \Gamma_1(\gamma_1)\Bigr)\cap\Bigl(\Delta\cap \Gamma_1(\gamma_2)\Bigr)=\emptyset$\end{center} for $\gamma_1,\gamma_2\in B_j$ with $\gamma_1\neq \gamma_2$. By the definition of $b_s'$, for each $\gamma\in B_j$ the size
\begin{equation}
|\Delta \cap\Gamma_1(\gamma)|=b_s'.
\end{equation}
Then $\Delta_j$ is the disjoint union of all the set $\Delta\cap\Gamma_1(\gamma)$ for $\gamma\in B_j$ and
\begin{center}$|\Delta_j|= b_s'|B_j|=b_s'(k-1)^{s-1}\geqslant 1$\end{center}
By the discussion of the last paragraph, we have that the vertex stabilizer $G_\alpha$ is transitive on the set $\{\Delta_j \mid j\in X\}$.

Let $S\subseteq X$ with $|S|\leqslant c_s'$. Define
\begin{eqnarray*}
\beta(S)&=&\bigcap\limits_{j\in S}\Gamma_s(\beta_j),\\
\Delta(S)&=&\bigcap\limits_{j\in S}\Delta_j=\Delta\cap\left(~\bigcap\limits_{j\in S}\Gamma_s(\beta_j)\right)=\Delta\cap \beta(S),\\
B_j(S)&=&\bigcup\limits_{\delta\in\Delta(S)}B_j\cap\Gamma_1(\delta).
\end{eqnarray*}
Sometimes, when we want to specify $\Delta$, we use $B_j(S)_\Delta$ to denote $B_j(S)$. Then
\begin{eqnarray*}
\Delta(S)&=&\{\delta\in \Delta\mid \text{for each } j\in S,~\Gamma_1(\delta)\cap B_j\neq\emptyset\}\\
&=&\{\delta\in \Delta\mid \text{for each } j\in S,~|\Gamma_1(\delta)\cap B_j|=1\}\\
&=&\{\delta\in \Delta\mid \text{for each } j\in S,~\partial_\Gamma(\delta,\beta_j)=s\}.
\end{eqnarray*}
Let $S$ and $T$ be two subsets of $X$ with $S\subseteq T$. Then $\Delta(S)\supseteq \Delta(T)$ and $\Delta(S)\subseteq \Delta_j$ for any $j\in S$.



Let $S\subseteq X$ with $|S|=c_s'$. Suppose $\Delta(S)\neq\emptyset$, and let $x\in \Delta(S)$. Then
\begin{equation}\label{eqn Gamma1x cap Bj}
|\Gamma_1(x)\cap B_j|=\left\{\begin{array}{cl} 1, &  \text{ if $j\in S$;}\\ 0, &  \text{ if $j\in X\setminus S$.}\end{array}\right.
\end{equation}
Let
\begin{eqnarray*}
S(x)&=&\{j\in X\mid \Gamma_1(x)\cap B_j\neq\emptyset\}\\
&=&\{j\in X\mid |\Gamma_1(x)\cap B_j|=1\}\\
&=&\{j\in X\mid \partial_\Gamma(x,\beta_j)=s\}.
\end{eqnarray*}
Then $S=S(x)$. Let $j\in S$. Then for each $\gamma\in B_j(S)$, we have
\begin{equation}
1\leqslant |\Gamma_1(\gamma)\cap \Delta(S)|\leqslant b_s'.
\end{equation}

Let $1\leqslant u\leqslant c_s'$ and let
\begin{equation}\label{equation B S subset X}
\mathcal{B}(u)=\{S\subseteq X\mid |S|=u\text{ and }\Delta(S)\neq\emptyset\}.
\end{equation}
Then $\mathcal{B}(1)=\{\{j\}\mid j\in X\}$ since $\Delta_j\neq \emptyset$ for each $j\in X$, and
\begin{equation}\label{eqn Delta = union of DeltaS S in Bu}
\Delta=\bigcup\limits_{S\in \mathcal{B}(u)}\Delta(S).
\end{equation}
The set $\mathcal{B}(c_s')$ has the following properties.

\begin{lemma}\label{lem DeltaiS DeltaiT}
Let $S,T\in \mathcal{B}(c_s')=\{S\subseteq X\mid |S|=c_s'\text{ and }\Delta(S)\neq\emptyset\}$.
\begin{enumerate}
\item If $S\neq T$, then \begin{equation}\label{eqn Delta S cap Delta T}\Delta(S)\cap \Delta(T)=\emptyset.\end{equation}
\item Let $x\in \Delta(S)$ and $y\in \Delta(T)$, and let $g\in G_\alpha$ such that $y=x^g$. Then $\Delta^g(S)=\Delta(T)$, and $g$ induces a bijection from $S$ to $T$.
\end{enumerate}
\end{lemma}

\begin{proof}
Suppose $S\neq T$ and $\Delta(S)\cap \Delta(T)\neq\emptyset$. Let $x\in \Delta(S)\cap \Delta(T)$. Then $|\Gamma_1(x)\cap B_j|=1$ for each $j\in S\cup T$. So $c_s'=|\Gamma_1(x)\cap \Gamma_s(\alpha)|=\left|\bigcup\limits_{j\in S\cup T}\Gamma_1(x)\cap B_j\right|=\sum\limits_{j\in S\cup T}|\Gamma_1(x)\cap B_j|=|S\cup T| > |S|=|T|=c_s'$. This is a contradiction.

Let $x\in \Delta(S)$ and $y\in \Delta(T)$. Since $G_\alpha$ is transitive on $\Delta$, there exists an element $g\in G_\alpha$ such that $y=x^g$. Let $g$ be such an automorphism. Then $g$ maps $(\alpha,x)$-paths of length $s+1$ to $(\alpha,y)$-paths of length $s+1$. Then $g$ maps $L_x=\Gamma_1(x)\cap \Gamma_s(\alpha)$ to $L_y=\Gamma_1(y)\cap \Gamma_s(\alpha)$. Since $g$ is an automorphism, the restriction $g|_{L_x}$ of $g$ on $L_x$ is injective. The numbers $|L_x|=|L_y|=c_s'$. So $g|_{L_x}$ is a one to one correspondence from $L_x$ to $L_y$. Suppose $L_x=\{x_1,x_2,\ldots,x_c\}$ where $c=c_s'$. Then $L_y=L_x^g=\{y_d=x_d^g\mid d=1,2,\ldots,c\}$. Suppose for $1\leqslant d\leqslant c$, the vertex $x_d\in B_{j_d}$ for some $j_d\in X$. Then $S=\{j_1,j_2,\ldots,j_c\}$. For $1\leqslant d\leqslant c$, the vertex $y_d=x_d^g\in B_{j_d}^g$. There exists exactly one $(\alpha,x)$-path $P_d$ of length $s+1$ such that $x_d$ lies on $P_d$, and there exists exactly one $(\alpha,y)$-path $Q_d$ of length $s+1$ such that $y_d$ lies on $Q_d$. Then $\beta_{j_d}$ lies on $P_d$ and $g$ maps $P_d$ to $Q_d$. Let $\beta_{h_d}\in \Gamma_1(\alpha)$ lie on $Q_d$ for some $h_d\in X$. Then $\beta_{h_d}=\beta_{j_d}^g$ and $y_d\in B_{h_d}$. Hence $T=\{h_1,h_2,\ldots,h_c\}$. The element $g$ maps $\Gamma_s(\alpha)\cap \Gamma_{s-1}(\beta_{j_d})=B_{j_d}$ to $\Gamma_s(\alpha)\cap \Gamma_{s-1}(\beta_{h_d})=B_{h_d}$. Since $|B_{j_d}|=|B_{h_d}|$, we have $B_{j_d}^g=B_{h_d}$. For any $z\in \Delta(S)$, let $\Gamma_1(z)\cap \Gamma_s(\alpha)=\{z_1,z_2\ldots,z_c\}$ and $z_d\in B_{j_d}$ for each $1\leqslant d\leqslant c$. Then $z_d^g\in B_{j_d}^g=B_{h_d}$. Since $z^g\in \Delta$ and $\Gamma_1(z^g)\cap \Gamma_s(\alpha)=\{z_1^g,z_2^g\ldots,z_c^g\}$, we have $z^g\in \Delta(T)$. Thus $\Delta^g(S)\subseteq \Delta(T)$. By the same argument as above, we have $\Delta^{g^{-1}}(T)\subseteq \Delta(S)$. Since $g$ is an automorphism, $\Delta(T)=\Delta^{g^{-1}g}(T)\subseteq \Delta^g(S)$. Hence $\Delta^g(S)=\Delta(T)$.
\end{proof}

Let $S,T\in \mathcal{B}(c_s')$. As a consequence of Lemma~\ref{lem DeltaiS DeltaiT}, we have $|\Delta(S)|=|\Delta(T)|$. This means that the size
\begin{equation}\label{eqn ei}
\e =|\Delta(S)|
\end{equation}
is independent of the choice of $S\in \mathcal{B}(c_s')$. Let
\begin{equation}
\mathcal{P}(u)=\{\Delta(S)\mid S\in \mathcal{B}(u)\}.
\end{equation}
Then $\mathcal{P}(c_s')$ is a partition of $\Delta$ by Lemma~\ref{lem DeltaiS DeltaiT} and Equation~(\ref{eqn Delta = union of DeltaS S in Bu}), and so $|\mathcal{B}(c_s')|\e =|\Delta|$.
Lemma~\ref{lem DeltaiS DeltaiT} also implies that $G_\alpha$ is transitive on $\mathcal{P}(c_s')$. For $1\leqslant v\leqslant k$, let
\begin{eqnarray}\label{eqn B(j)}
\mathcal{B}(u,v) & = & \{S\in \mathcal{B}(u)\mid v\in S\}\nonumber \\
& = & \{S\subseteq X\mid v\in S, |S|=u, \Delta(S)\neq\emptyset\}.
\end{eqnarray}
Then
\begin{eqnarray}
\Delta_j & = & \bigcup\limits_{S\in \mathcal{B}(u,j)}\Delta(S),\label{eqn Deltaj = union of DeltaS S in Buj} \\
B_j & = & \bigcup\limits_{S\in \mathcal{B}(u,j)}B_j(S).\label{eqn Bj union of BjS over S in Buj}
\end{eqnarray}
Let
\begin{eqnarray}
\mathcal{P}(u,v) & = & \{\Delta(S)\mid S\in \mathcal{B}(u,v)\}\\
B_v^{\mathcal{P}(u)} & = & \{B_v(S)\mid S\in \mathcal{B}(u,v)\}.
\end{eqnarray}
Then $\mathcal{P}(c_s',j)=\{\Delta(S)\mid S\in \mathcal{B}(c_s',j)\}$ is a partition of $\Delta_j$ by Lemma~\ref{lem DeltaiS DeltaiT} and Equation~(\ref{eqn Deltaj = union of DeltaS S in Buj}), and so $|\mathcal{B}(c_s',j)|\e=|\Delta_j|$.

\begin{corollary}\label{cor P(cs') is a partition of Delta P(cs',j) is a partition of Deltaj}
Let $j\in X$. Then the set $\mathcal{P}(c_s')$ is a partition of $\Delta$, and the set $\mathcal{P}(c_s',j)$ is a partition of $\Delta_j$. We have the following equations
\begin{eqnarray}
|\mathcal{B}(c_s')|\e & = & |\Delta|={\displaystyle \frac{b_s'}{c_s'}}k(k-1)^{s-1},\label{eqn ei B = Deltai}\\
|\mathcal{B}(c_s',j)|\e & = & |\Delta_j|=b_s'(k-1)^{s-1}.\label{eqn ei Bj = Deltaj}
\end{eqnarray}
\end{corollary}

Let
\begin{eqnarray*}
S\in \mathcal{B}(c_s',j)&=&\{S\in \mathcal{B}(c_s')\mid j\in S\}\\
&=&\{S\subseteq X\mid j\in S, |S|=c_s', \Delta(S)\neq\emptyset\}.
\end{eqnarray*}
Then $1\leqslant |\Gamma_1(\gamma)\cap \Delta(S)|\leqslant b_s'$ for each $\gamma\in B_j(S)$. We now consider the general case. Let $j\in X$ and let $\gamma\in B_j$. Let
\begin{eqnarray*}
B_j^\gamma&=&\{S\in \mathcal{B}(c_s',j)\mid \Gamma_1(\gamma)\cap \Delta(S)\neq\emptyset\},\\
\Delta_j^\gamma&=&\{\Gamma_1(\gamma)\cap \Delta(S)\mid S\in B_j^\gamma\}.
\end{eqnarray*}
Then $\Delta_j^\gamma$ is a partition of $\Gamma_1(\gamma)\cap \Delta_j=\Gamma_1(\gamma)\cap \Delta$ (see Figure~\ref{fig: partitionofbs}) and
\begin{equation}\label{eqn bs'= sum of Gamma1gamma cap DeltaS}
b_s'=\sum\limits_{S\in B_j^\gamma}|\Gamma_1(\gamma)\cap \Delta(S)|.
\end{equation}

\begin{lemma}\label{lem Delta j gamma 1 to 1 Delta h varepsilon}
Let $\gamma\in B_j$ and $\varepsilon\in B_h$, and let $g\in G_\alpha$ such that $\varepsilon=\gamma^g$. Then $g$ induces a bijection from $\Delta_j^\gamma$ to $\Delta_h^\varepsilon$, and $g$ induces a bijection from $B_j^\gamma$ to $B_h^\varepsilon$.
\end{lemma}

\begin{proof}
This automorphism $g$ also maps $\beta_j$ to $\beta_h$, and so $\Delta_j^g=\Delta_h$. Then \begin{center}$(\Gamma_1(\gamma)\cap \Delta_j)^g\subseteq \Gamma_1^g(\gamma)\cap \Delta_j^g=\Gamma_1(\gamma^g)\cap \Delta_h\newline =\Gamma_1(\varepsilon)\cap \Delta_h=(\Gamma_1(\varepsilon)\cap \Delta_h)^{g^{-1}g}\subseteq (\Gamma_1(\gamma)\cap \Delta_j)^g$.\end{center} This means $(\Gamma_1(\gamma)\cap \Delta_j)^g=\Gamma_1(\varepsilon)\cap \Delta_h$. By the same argument as above, for each $S\in B_j^\gamma$, we have $(\Gamma_1(\gamma)\cap \Delta(S))^g=\Gamma_1^g(\gamma)\cap \Delta^g(S)=\Gamma_1(\varepsilon)\cap \Delta(T)$ where $T\in B_h^\varepsilon$ and $\Delta^g(S)=\Delta(T)$. We know
\begin{eqnarray*}
\Gamma_1(\gamma)\cap \Delta_j&=&\bigcup\limits_{S\in B_j^\gamma}\Gamma_1(\gamma)\cap \Delta(S),\\
\Gamma_1(\varepsilon)\cap \Delta_h&=&\bigcup\limits_{T\in B_h^\varepsilon}\Gamma_1(\varepsilon)\cap \Delta(T).
\end{eqnarray*}
Hence $(\Gamma_1(\gamma)\cap \Delta_j)^g=\bigcup\limits_{S\in B_j^\gamma}(\Gamma_1(\gamma)\cap \Delta(S))^g=\bigcup\limits_{T\in B_h^\varepsilon}\Gamma_1(\varepsilon)\cap \Delta(T)=\Gamma_1(\varepsilon)\cap \Delta_h$. This completes the proof.
\end{proof}

Let $\gamma\in \Gamma_s(\alpha)$. Then $\gamma\in B_j$ for some $j\in X$. By Lemma~\ref{lem Delta j gamma 1 to 1 Delta h varepsilon}, the multiset $\left[|\Gamma_1(\gamma)\cap \Delta(S)|~\Bigl|~ S\in B_j^\gamma\right]$ is independent of the choice of $\gamma$ and $j$. Define
\begin{equation}
P_{s+1}(\alpha,\Delta)=\left[|\Gamma_1(\gamma)\cap \Delta(S)|~\Bigl|~ S\in B_j^\gamma\right].
\end{equation}
By Equation~(\ref{eqn bs'= sum of Gamma1gamma cap DeltaS}), we know $P_{s+1}(\alpha,\Delta)$ corresponds to a partition of the integer $b_s'$. If $b_s'=1$, then $P_{s+1}(\alpha,\Delta)=[1]$. If $b_s'=2$, then $P_{s+1}(\alpha,\Delta)=[2]$ or $[1,1]$. If for any $S\in B_j^\gamma$, we have $|\Gamma_1(\gamma)\cap\Delta(S)|=b$, then we have $b_s'=mb$ and $P_{s+1}(\alpha,\Delta)=[b,b\ldots,b]$ where $m=|B_j^\gamma|$. We will consider the adjacency relations between $\Gamma_s(\alpha)$ and $\Delta$ of these special cases of $P_{s+1}(\alpha,\Delta)$ in Section~\ref{sec Adjacency relations}.

\begin{figure}[htb]
\centering
\includegraphics[width=0.80\textwidth]{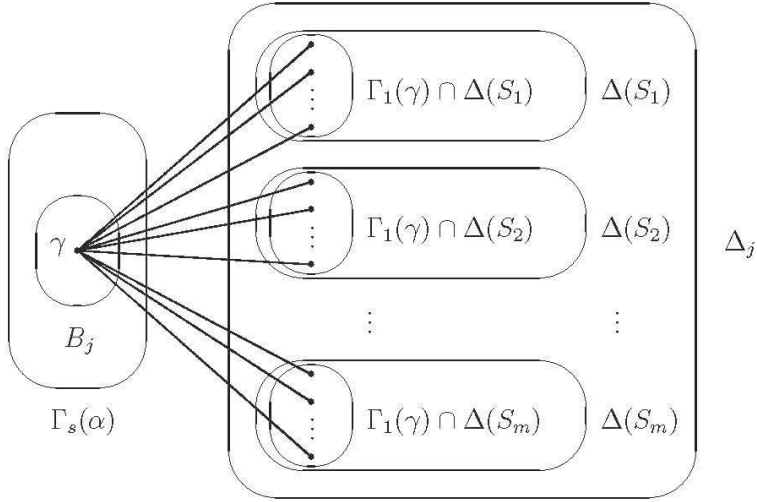}
\caption{The distribution of $\Gamma_1(\gamma)$ in $\Delta_j$ where $\gamma\in B_j$, $m=|B_j^\gamma|$ and $B_j^\gamma=\{S_1,S_2,\ldots,S_m\}$.}
\label{fig: partitionofbs}
\end{figure}

\begin{lemma}\label{lem DeltaiTt DeltaiT't}
Let $T$ and $T'$ be $t$-subsets of $X$. Then there exists $g\in G_\alpha$ such that $\Delta(T)^g=\Delta(T')$.
\end{lemma}

\begin{proof}
Suppose $T=\{j_1,j_2,\ldots,j_t\}$ and $T'=\{j_1',j_2',\ldots,j_t'\}$. Since the local action of $G_\alpha$ on $\Gamma_1(\alpha)$ is $t$-homogeneous, there exists an element $g\in G_\alpha$ such that $\{\beta_j\mid j\in T\}^g=\{\beta_j\mid j\in T'\}$. Without loss of generality, we may assume $\beta_{j_h}^g=\beta_{j_h'}$ for $1\leqslant h\leqslant t$. So $\Delta^g=\Delta$ and $\Gamma_s(\beta_{j_h})^g=\Gamma_s(\beta_{j_h}^g)=\Gamma_s(\beta_{j_h'})$ for $1\leqslant h\leqslant t$. By definition, we have $\Delta(T)^g=\left(\Delta\cap \left(\bigcap\limits_{1\leqslant h\leqslant t}\Gamma_s(\beta_{j_h})\right)\right)^g\subseteq \Delta^g\cap \left(\bigcap\limits_{1\leqslant h\leqslant t}\Gamma_s(\beta_{j_h})^g\right)=\Delta\cap \left(\bigcap\limits_{1\leqslant h\leqslant t}\Gamma_s(\beta_{j_h'})\right)=\Delta(T')$, that is $\Delta(T)^g\subseteq \Delta(T')$. By the same argument, we have $\Delta(T')^{g^{-1}}\subseteq \Delta(T)$. Since $g$ is an automorphism, $\Delta(T')=\Delta(T')^{g^{-1}g}\subseteq \Delta(T)^g$. Hence $\Delta(T)^g=\Delta(T')$.
\end{proof}

By Lemma~\ref{lem DeltaiTt DeltaiT't}, we have $|\Delta(T)|=|\Delta(T')|$ for any $t$-subsets $T$ and $T'$ of $X$. This means that the size
\begin{equation}\label{eqn deltait}
\delta_t=|\Delta(T)|
\end{equation}
is independent of the choice of $T$. Note that $c_s'\geqslant t$. Take $x\in \Delta$. Let $S=S(x)$. Then $|S|=c_s'$ and $x\in \Delta(S)\neq\emptyset$. Let $T\subseteq S$ with $|T|=t$. Then $\Delta(S)\subseteq \Delta(T)\neq\emptyset$. By Lemma~\ref{lem DeltaiTt DeltaiT't}, for any $t$-subset $T'$ of $X$ we have $|\Delta(T')|=|\Delta(T)|\neq 0$ and so $\Delta(T')\neq\emptyset$. Hence $\delta_t\neq 0$ and
\begin{eqnarray*}
\mathcal{B}(t)&=&\{T\subseteq X\mid |T|=t\text{ and }\Delta(T)\neq\emptyset\}\\
&=&\{T\subseteq X\mid |T|=t\}.
\end{eqnarray*}
Let $1\leqslant j\leqslant t$. Since $t$-homogeneous implies $j$-homogeneous, by the same argument as above, we have $\mathcal{B}(j)=\{T\subseteq X\mid |T|=j\}.$

Here we use the convention that $\binom{r}{0}=1$ for any $r\geqslant 0$.

\begin{lemma}\label{lem deltait}
Let $\delta_t$ be as in Equation~\ref{eqn deltait}. Then
\begin{equation}
\delta_t=|\Delta|\binom{c_s'}{t}\biggl/\binom{k}{t}=b_s'(k-1)^{s-1}\binom{c_s'-1}{t-1}\biggl/\binom{k-1}{t-1}.
\end{equation}
\end{lemma}

\begin{proof}
Let \begin{center}$Y_t=\Bigl\{(\delta,\Delta(T))\mid \delta\in \Delta, T\in \mathcal{B}(t), \text{ and }\delta\in \Delta(T)\Bigr\}$.\end{center} We will count the size of $Y_t$ in two ways. For each $T\in \mathcal{B}(t)$, the set $\Delta(T)$ has $\delta_t$ elements. So \begin{center}$\#Y_t=\delta_t|\mathcal{B}(t)|=\delta_t\binom{k}{t}$.\end{center} For any $\delta\in \Delta$, the set $S(\delta)=\{j\in X\mid \partial_\Gamma(\delta,\beta_j)=s\}$ has $c_s'$ elements. Let $S=S(\delta)$. So $\delta\in \Delta(S)$ and there are $\binom{c_s'}{t}$ elements $T\in \mathcal{B}(t)$ such that $T\subseteq S$ which implies $\Delta(S)\subseteq \Delta(T)$. Hence $\#Y_t=|\Delta|\binom{c_s'}{t}$. By counting edges between $\Gamma_s(\alpha)$ and $\Delta$, we have
\begin{equation}
|\Delta|c_s'=|\Gamma_s(\alpha)|b_s'=k_sb_s'=b_s'k(k-1)^{s-1}.
\end{equation}
So \begin{center}$\#Y_t=|\Delta|\binom{c_s'}{t}=b_s'k(k-1)^{s-1}\binom{c_s'}{t}/c_s'$.\end{center} Thus $\delta_t\binom{k}{t}=\#Y_t=b_s'k(k-1)^{s-1}\binom{c_s'}{t}/c_s'$. So $\delta_t=b_s'k(k-1)^{s-1}\binom{c_s'}{t}\Bigl/\left(c_s'\binom{k}{t}\right)=b_s'(k-1)^{s-1}\binom{c_s'-1}{t-1}\Bigl/\binom{k-1}{t-1}$.
\end{proof}

Let $T\subseteq X$ with $|T|\leqslant u$ and let
\begin{eqnarray}
\mathcal{B}(u;T)&=&\bigcap\limits_{j\in T}\mathcal{B}(u,j)=\{S\in \mathcal{B}(u)\mid T\subseteq S\},\\
\mathcal{P}(u;T)&=&\bigcap\limits_{j\in T}\mathcal{P}(u,j)=\{\Delta(S)\mid S\in \mathcal{B}(u;T)\}\nonumber\\
&=&\{\Delta(S)\mid T\subseteq S\in \mathcal{B}(u)\}.
\end{eqnarray}
For any $S\in \mathcal{B}(u;T)$, we know $\Delta(S)\subseteq \Delta(T)$.
Then
\begin{equation}
\Delta(T)=\bigcup\limits_{S\in \mathcal{B}(u;T)}\Delta(S).
\end{equation}
Hence the set $\mathcal{P}(c_s';T)$ forms a partition of $\Delta(T)$ by Lemma~\ref{lem DeltaiS DeltaiT}.


\begin{theorem}\label{theorem t-design}
Let $T$ be a $t$-subset of $X$. Then the size $\lambda_t=|\mathcal{B}(c_s';T)|$ is independent of the choice of $T$. As a consequence, $(X,\mathcal{B}(c_s'))$ is a $t$-$(k,c_s',\lambda_t)$ design where
\begin{equation}\label{eqn lambdait ei = deltait}
\lambda_t\e =\delta_t=b_s'(k-1)^{s-1}\binom{c_s'-1}{t-1}\biggl/\binom{k-1}{t-1}.
\end{equation}
In particular, if $t=c_s'$, then $\mathcal{B}(c_s')$ consists of all $c_s'$-subsets of $X$, $\lambda_{c_s'}=1$, and $(X,\mathcal{B}(c_s'))$ is a $c_s'$-$(k,c_s',1)$ design.
\end{theorem}

\begin{proof}
Since $\mathcal{P}(c_s';T)$ forms a partition of $\Delta(T)$, we have $\delta_t=|\Delta(T)|=\left|\bigcup\limits_{S\in \mathcal{B}(c_s';T)}\Delta(S)\right|=\sum\limits_{S\in \mathcal{B}(c_s';T)}|\Delta(S)|=|\mathcal{B}(c_s';T)|\e$. So the size $\lambda_t=|\mathcal{B}(c_s';T)|$ is independent of the choice of $T$ and $\lambda_t\e =\delta_t=b_s'(k-1)^{s-1}\binom{c_s'-1}{t-1}\Bigl/\binom{k-1}{t-1}$ by Lemma~\ref{lem deltait}. Thus $(X,\mathcal{B}(c_s'))$ is a $t$-$(k,c_s',\lambda_t)$ design.

Suppose $t=c_s'$. Then $\mathcal{B}(c_s')=\mathcal{B}(t)=\{S\subseteq X\mid |S|=t\}$. We use the equations of $\#Y_t$ in the proof of Lemma~\ref{lem deltait} and Equation~\ref{eqn ei B = Deltai}. We have $\delta_t|\mathcal{B}(t)|=\# Y_t=|\Delta|\binom{c_s'}{t}=|\Delta|=|\mathcal{B}(c_s')|\e$, and so $\delta_t=\e$. Hence $\lambda_t=1$.
%
\end{proof}


In Theorem~\ref{theorem t-design}, we say that the $t$-$(k,c_s',\lambda_t)$ design $D=(X,\mathcal{B}(c_s'))$ is relative to the $G_\alpha$-orbit $\Delta$, or we write $D(\Delta)=(X,\mathcal{B}(c_s'))$ to specify the $G_\alpha$-orbit $\Delta$. The following is a direct consequence of Theorem~\ref{theorem t-design}.

\begin{corollary}\label{cor t=1 1-design}
If $t=1$, then $(X,\mathcal{B}(c_s'))$ is a $1$-$(k,c_s',\lambda_1)$ design where
\begin{equation}\label{eqn lambda1 e = delta1}
\lambda_1\e =\delta_1=|\Delta_j|=b_s'(k-1)^{s-1}.
\end{equation}
\end{corollary}

Table~\ref{table 1-design for valency 3} and Table~\ref{table 1-design for valency 4} give all $1$-designs $(X,\mathcal{B}(c_s'))$, where $|X|=3$ or $4$, according to different structures of the graph $\Gamma$.

\begin{table}[!ht]
\centering
    \begin{tabular}{|c|c|c|c|}
        \hline
        $(X,\mathcal{B}(c_s'))$ & $|\mathcal{B}(c_s')|$ & $\mathcal{B}(c_s')$ & remark \\
        \hline
        $1$-$(3,1,1)$ design & $3$ & $\{\{1\},\{2\},\{3\}\}$ & $\backslash$ \\
        \hline
        $1$-$(3,2,2)$ design & $3$ & $\{\{1,2\},\{1,3\},\{2,3\}\}$ & $2$-$(3,2,1)$ design \\
        \hline
        $1$-$(3,3,1)$ design & $1$ & $\{\{1,2,3\}\}$ & $3$-$(3,3,1)$ design \\
        \hline
    \end{tabular}
    \caption{$1$-$(3,c_s',\lambda_1)$ design $(X,\mathcal{B}(c_s'))$ where $c_s'\in X=\{1,2,3\}$}\label{table 1-design for valency 3}
\end{table}

\begin{table}[!ht]
\centering
    \begin{tabular}{|c|c|c|c|}
        \hline
        $(X,\mathcal{B}(c_s'))$ & $|\mathcal{B}(c_s')|$ & $\mathcal{B}(c_s')$ & remark\\
        \hline
        $1$-$(4,1,1)$ design & $4$ & $\{\{1\},\{2\},\{3\},\{4\}\}$ & $\backslash$ \\
        \hline
        $1$-$(4,2,1)$ design & $2$ & \begin{tabular}{l}$\{\{1,2\},\{3,4\}\}$, \\ $\{\{1,3\},\{2,4\}\}$\\ or \\ $\{\{1,4\},\{2,3\}\}$ \end{tabular} & $\backslash$ \\
        \hline
        $1$-$(4,2,2)$ design & $4$ & \begin{tabular}{l}$\{\{1,2\},\{2,3\},\{3,4\},\{4,1\}\}$,\\ $\{\{1,2\},\{2,4\},\{4,3\},\{3,1\}\}$\\  or\\ $\{\{1,3\},\{3,2\},\{2,4\},\{4,1\}\}$\end{tabular} & $\backslash$ \\
        \hline
        $1$-$(4,2,3)$ design & $6$ & \begin{tabular}{l}$\{\{1,2\},\{1,3\},\{1,4\}$, \\ ~~~$\{2,3\},\{2,4\},\{3,4\}\}$\end{tabular} & \begin{tabular}{l}$2$-$(4,2,1)$\\ design\end{tabular} \\
        \hline
        $1$-$(4,3,3)$ design & $4$ & \begin{tabular}{l}$\{\{1,2,3\},\{1,2,4\}$,\\ ~~~$\{1,3,4\},\{2,3,4\}\}$\end{tabular} & \begin{tabular}{l}$3$-$(4,3,1)$\\ design\end{tabular} \\
        \hline
        $1$-$(4,4,1)$ design & $1$ & $\{\{1,2,3,4\}\}$ & \begin{tabular}{l}$4$-$(4,4,1)$\\ design\end{tabular} \\
        \hline
    \end{tabular}
    \caption{$1$-$(4,c_s',\lambda_1)$ design $(X,\mathcal{B}(c_s'))$ where $c_s'\in X=\{1,2,3,4\}$}\label{table 1-design for valency 4}
\end{table}



\section{Adjacency relations}\label{sec Adjacency relations}

In this section, we use notations and assumptions as in Section~\ref{sec Local actions and designs}. Here, we consider several special cases of $P_{s+1}(\alpha,\Delta)$ and give the corresponding adjacency relations.

\subsection{$P_{s+1}(\alpha,\Delta)=[b_s']$}\label{subsec Ps+1alphaDeltaisbs'}

First, when $b_s'=1$, we get the structural information between $\Gamma_s(\alpha)$ and $\Delta$ in the following lemma.

\begin{lemma}\label{lem structural information between Gammasalpha and Delta}
We suppose $b_s'=1$. Let $j\in X$ and $S\in \mathcal{B}(c_s',j)$. Then
\begin{enumerate}
\item the edges between $B_j$ and $\Delta_j$ is a matching,
\item the edges between $B_j(S)$ and $\Delta(S)$ is a matching,
\item $B_j^{\mathcal{P}(c_s')}$ forms a partition of $B_j$, and
\item $\mathcal{P}(c_s',j)$ is a partition of $\Delta_j$.
\end{enumerate}
\end{lemma}

\begin{proof}
We suppose $b_s'=1$. Let $j\in X$ and $S\in \mathcal{B}(c_s',j)$. Then $|\Delta_j|=b_s'|B_j|=|B_j|$. For each $\delta\in \Delta_j$ we have $|\Gamma_1(\delta)\cap B_j|=1$ by Equation~(\ref{eqn Gamma1x cap Bj}). For each $\gamma\in B_j$ we have $|\Gamma_1(\gamma)\cap \Delta_j|=b_s'=1$. So the edges between $B_j$ and $\Delta_j$ is a matching.

Let $\delta_1,\delta_2\in \Delta_j$ with $\delta_1\neq \delta_2$. We will show that
\begin{equation}\label{eqn disjoint union of BjS}
\Bigl(B_j\cap \Gamma_1(\delta_1)\Bigr)\cap \Bigl(B_j\cap \Gamma_1(\delta_2)\Bigr)=\emptyset.
\end{equation}
Let $A=\Bigl(B_j\cap \Gamma_1(\delta_1)\Bigr)\cap \Bigl(B_j\cap \Gamma_1(\delta_2)\Bigr)$. Suppose $A\neq \emptyset$ and take $\gamma\in A\subseteq B_j$. Then $\{\delta_1,\delta_2\}\subseteq \Gamma_1(\gamma)\cap \Delta_j$. We have $2\leqslant |\Gamma_1(\gamma)\cap \Delta_j|=b_s'=1$. This is a contradiction. So we get $A=\emptyset$. Hence $B_j(S)=\bigcup\limits_{\delta\in \Delta(S)}B_j\cap \Gamma_1(\delta)$ is the disjoint union of $B_j\cap \Gamma_1(\delta)$ over all $\delta\in \Delta(S)$. For each $\delta\in \Delta(S)$, we have $\emptyset\neq B_j\cap \Gamma_1(\delta)=B_j(S)\cap \Gamma_1(\delta)$ and $|B_j(S)\cap \Gamma_1(\delta)|=1$ by Equation~(\ref{eqn Gamma1delta cap Bj <=1}). For each $\gamma\in B_j(S)\subseteq B_j$, we have $0\neq |\Delta(S)\cap \Gamma_1(\gamma)|\leqslant |\Delta_j\cap \Gamma_1(\gamma)|=b_s'=1$, i.e. $|\Delta(S)\cap \Gamma_1(\gamma)|=1$. Thus the edges between $B_j(S)$ and $\Delta(S)$ is a matching.

Take $S,T\in \mathcal{B}(c_s',j)$ with $T\neq S$. Then $\Delta(S)\cup \Delta(T)\subseteq \Delta_j$. So by Equation~(\ref{eqn disjoint union of BjS}), we have
\begin{equation}\label{eqn BjS cap BjT empty}
B_j(S)\cap B_j(T)=\emptyset.
\end{equation}
By Equations~(\ref{eqn Bj union of BjS over S in Buj}) and (\ref{eqn BjS cap BjT empty}), we have that $B_j^{\mathcal{P}(c_s')}$ forms a partition of $B_j$.
\end{proof}

The adjacency relations between $B_j(S)$'s and $\Delta(S)$'s when $b_s'=1$ for valency $k=3$ and $k=4$ are shown in Figure~\ref{fig: Adjacency relations k3bi'1} and Figure~\ref{fig: Adjacency relations k4bi'1}. In these figures, we use $S\in \mathcal{B}(c_s')$ to denote the set $\Delta(S)$ in $\Delta$ and the corresponding set $B_j(S)$ in $B_j$, and we use an edge to denote the matching between $\Delta(S)$ and $B_j(S)$.

\begin{figure}[htb]
\centering
\subfigure[$c_s'=1$]{
\label{fig: k3bi'1ci'1}
\includegraphics[width=0.40\textwidth]{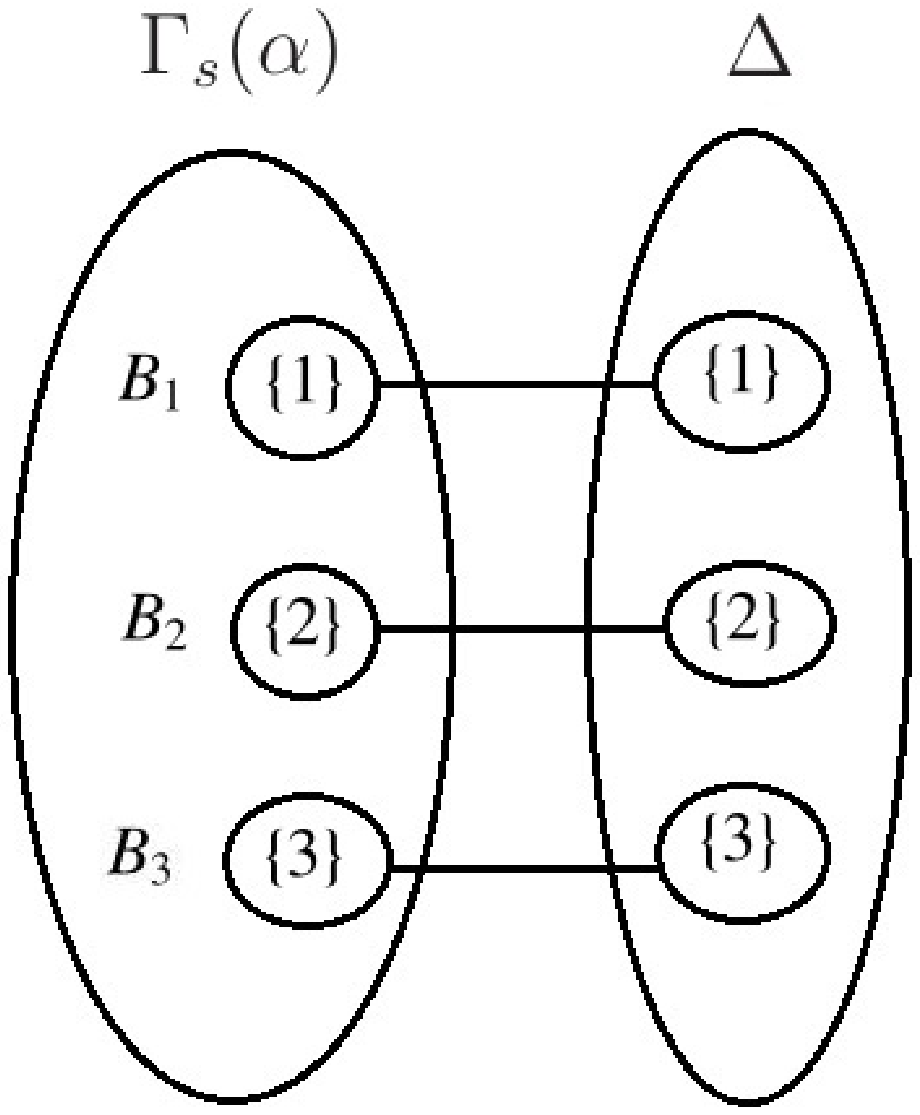}}
\subfigure[$c_s'=2$]{
\label{fig: k3bi'1ci'2}
\includegraphics[width=0.40\textwidth]{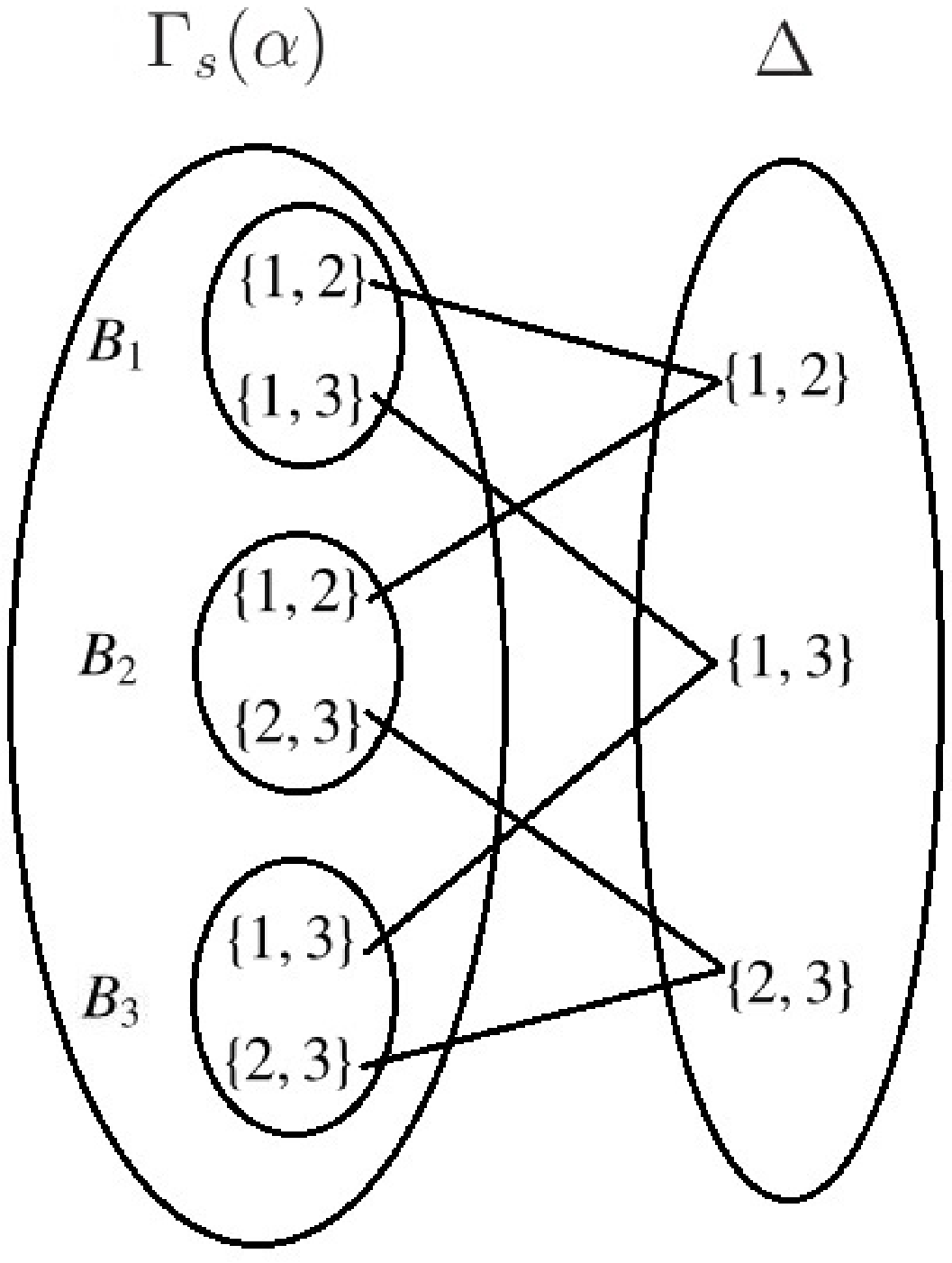}}
\subfigure[$c_s'=3$]{
\label{fig: k3bi'1ci'3}
\includegraphics[width=0.40\textwidth]{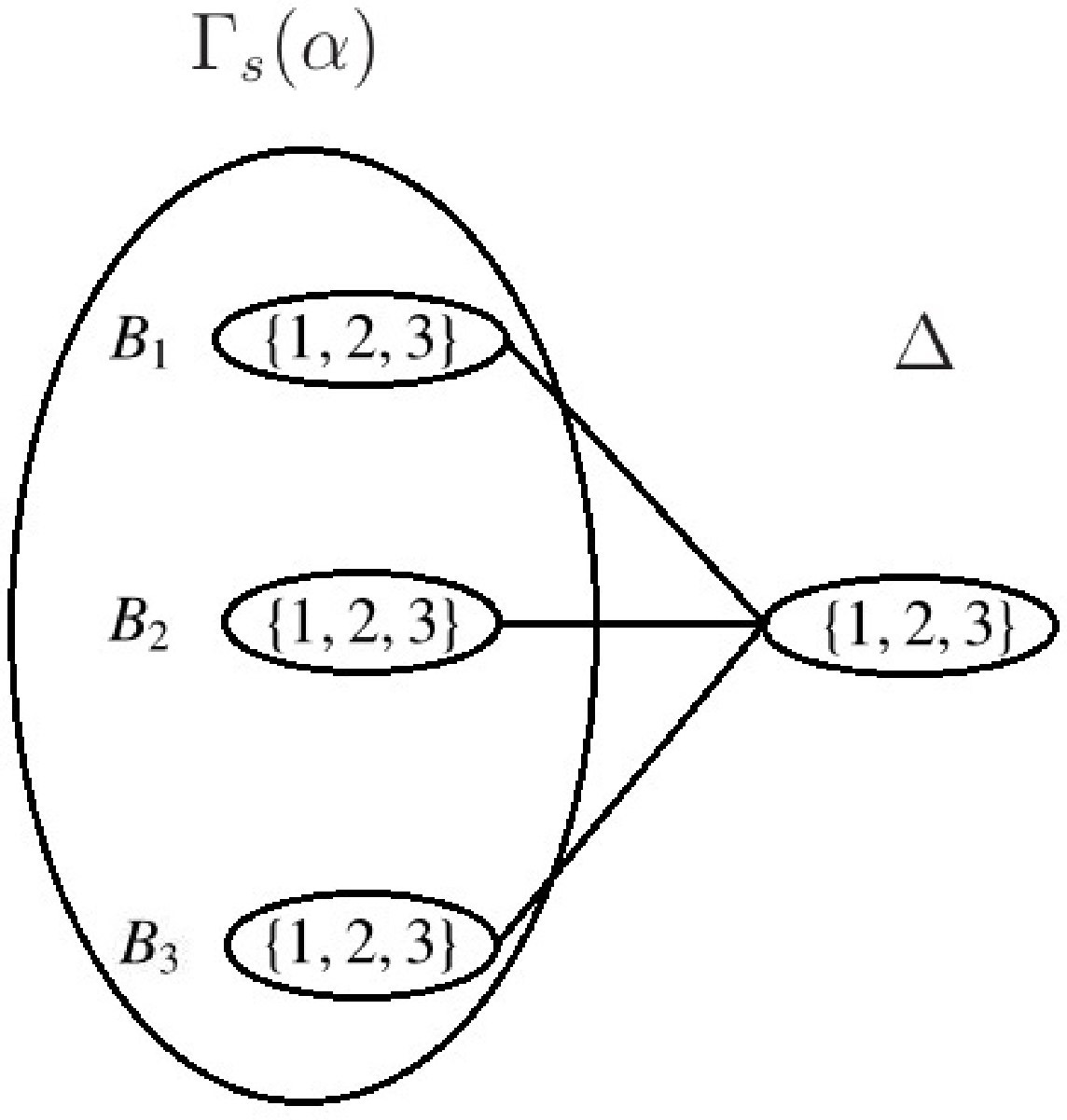}}
\caption{Adjacency relations when $b_s'=1$ and $k=3$}
\label{fig: Adjacency relations k3bi'1}
\end{figure}

\begin{figure}[htb]
\centering
\subfigure[$c_s'=1$]{
\label{fig: k4bi'1ci'1}
\includegraphics[width=0.31\textwidth]{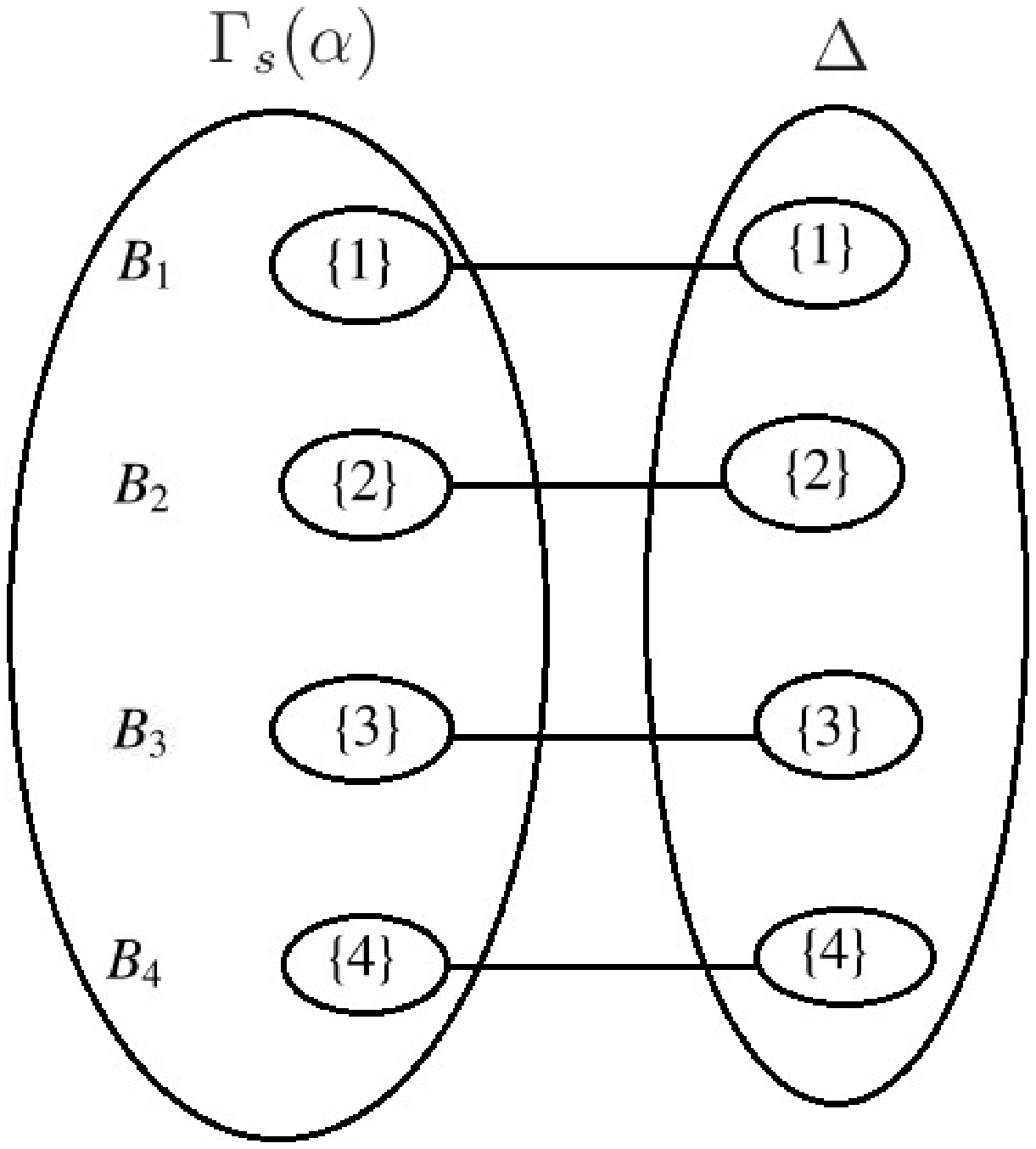}}
\subfigure[$c_s'=2$ and $\lambda_1=1$]{
\label{fig: k4bi'1ci'21}
\includegraphics[width=0.31\textwidth]{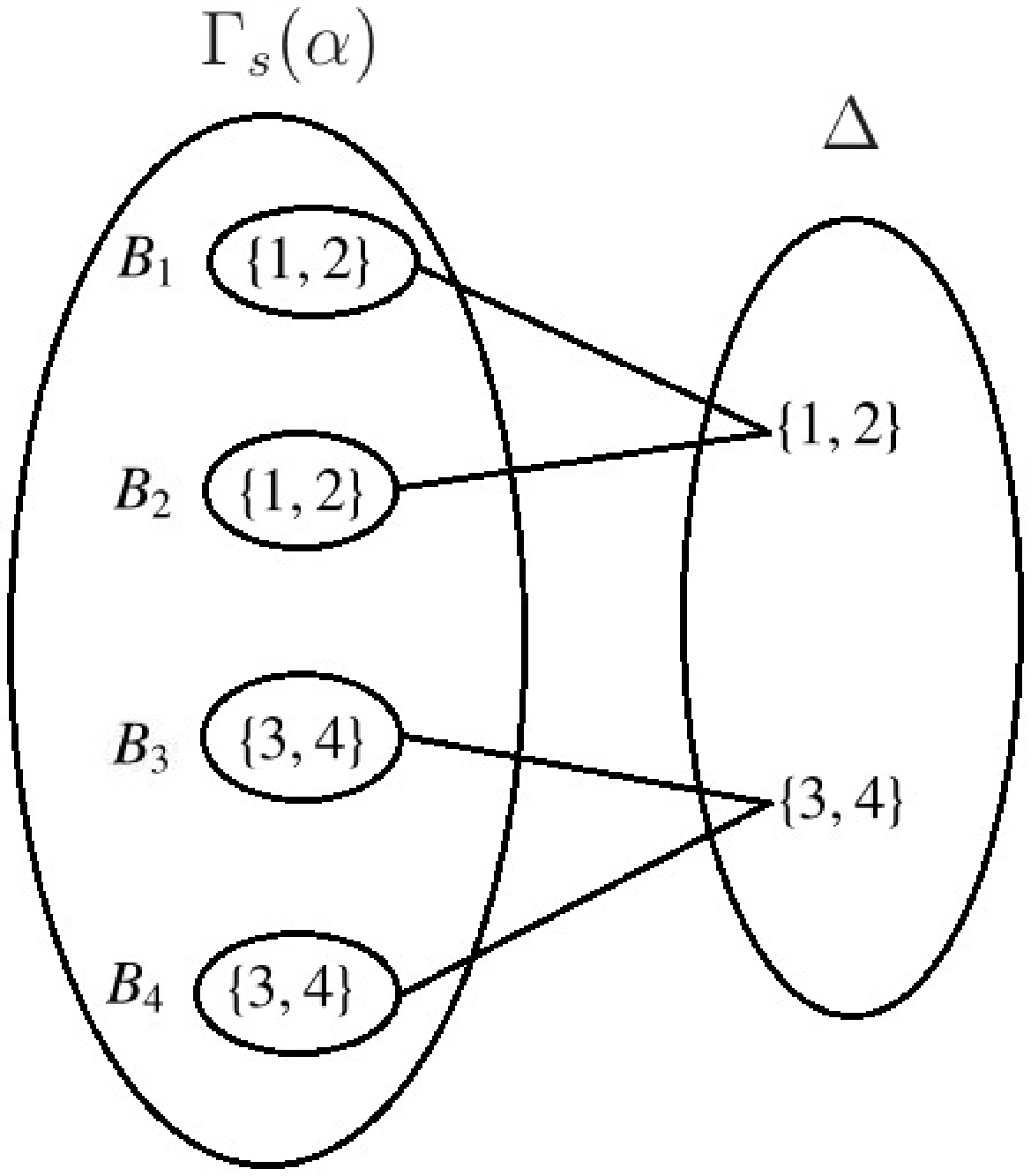}}
\subfigure[$c_s'=2$ and $\lambda_1=3$]{
\label{fig: k4bi'1ci'23}
\includegraphics[width=0.31\textwidth]{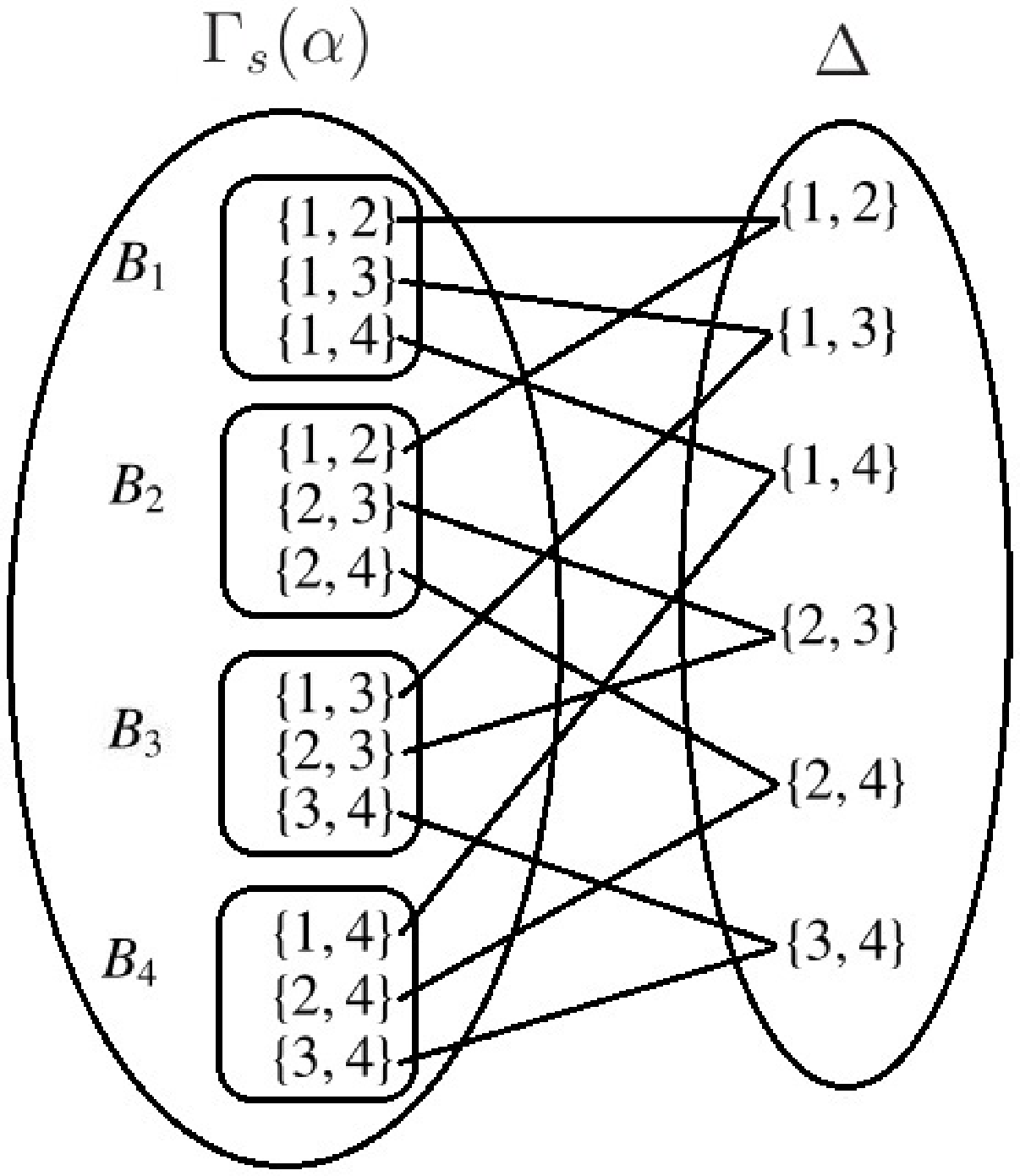}}
\subfigure[$c_s'=3$]{
\label{fig: k4bi'1ci'3}
\includegraphics[width=0.31\textwidth]{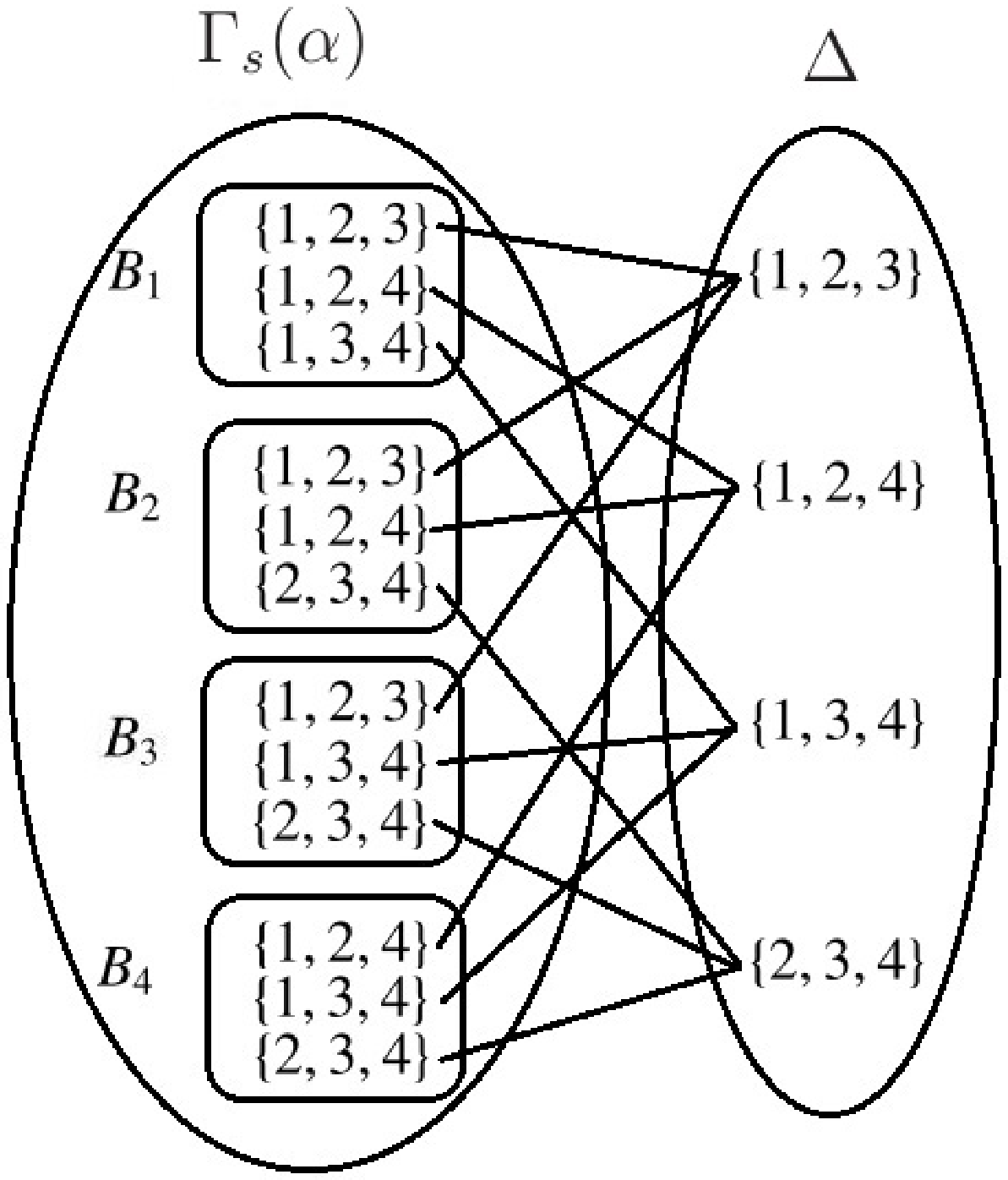}}
\subfigure[$c_s'=4$]{
\label{fig: k4bi'1ci'4}
\includegraphics[width=0.31\textwidth]{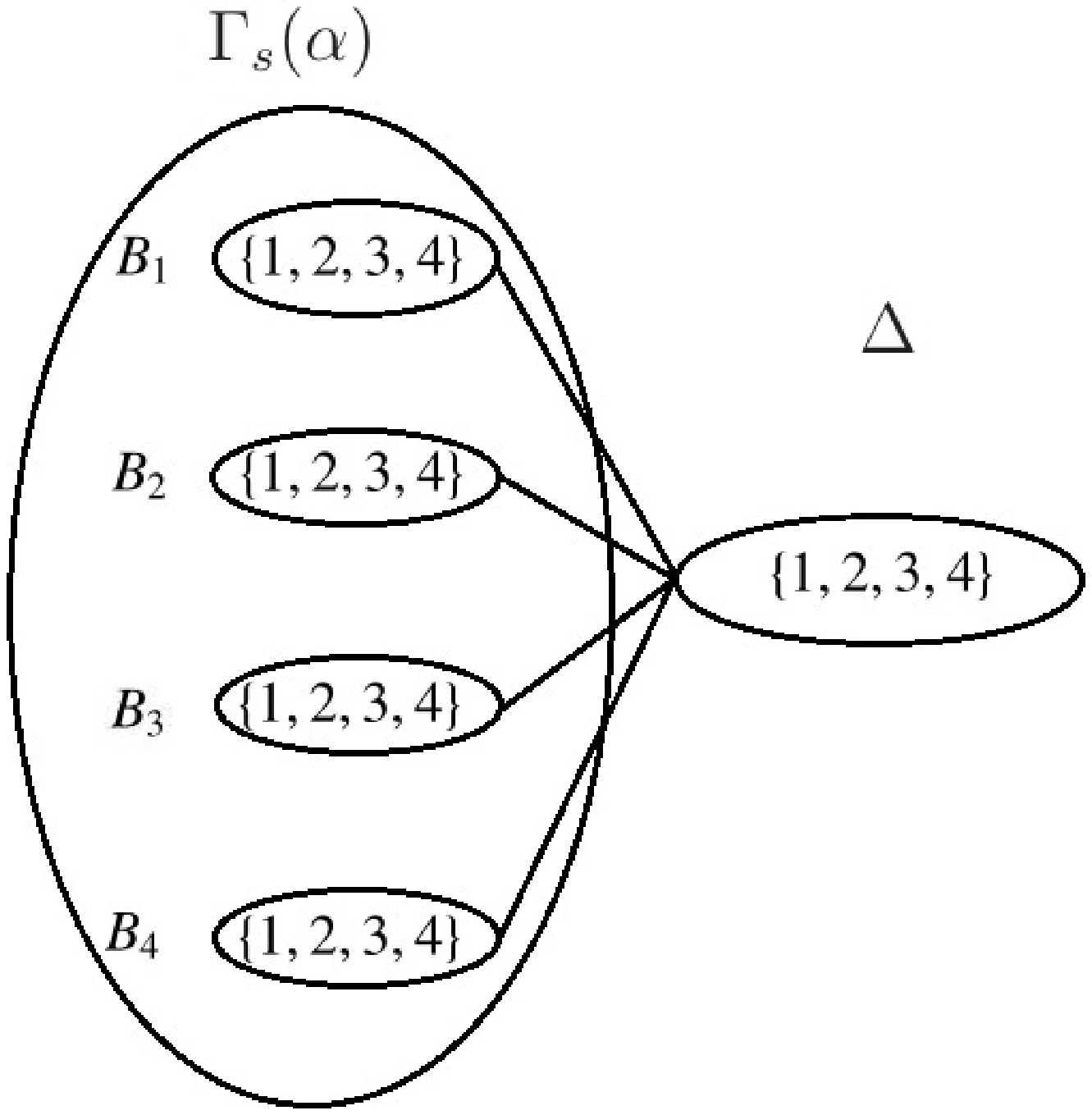}}
\caption{Adjacency relations when $b_s'=1$ and $k=4$}
\label{fig: Adjacency relations k4bi'1}
\end{figure}


Similarly, for the general case of $P_{s+1}(\alpha,\Delta)=[b_s']$ where $b_s'\geqslant 1$, we have the following lemma.

\begin{lemma}\label{lem structural information P_s+1=[bs']}
Suppose $P_{s+1}(\alpha,\Delta)=[b_s']$. Let $j\in X$ and let $S,T\in \mathcal{B}(c_s',j)$ with $S\neq T$. Then $B_j(S)\cap B_j(T)=\emptyset$. Hence
\begin{enumerate}
\item the edges between $B_j$ and $\Delta_j$ is a disjoint union of $K_{1,b_s'}$'s,
\item the edges between $B_j(S)$ and $\Delta(S)$ is a disjoint union of $K_{1,b_s'}$'s,
\item $B_j^{\mathcal{P}(c_s')}$ forms a partition of $B_j$, and
\item $\mathcal{P}(c_s',j)$ is a partition of $\Delta_j$.
\end{enumerate}
\end{lemma}

\begin{proof}
Suppose $B_j(S)\cap B_j(T)\neq\emptyset$ and let $\gamma\in B_j(S)\cap B_j(T)\subseteq B_j$. Then $\Gamma_1(\gamma)\cap \Delta(S)\neq \emptyset$ and $\Gamma_1(\gamma)\cap \Delta(T)\neq \emptyset$. This means $|\Gamma_1(\gamma)\cap \Delta(S)|<b_s'$. This contradicts $m_{s+1}(\alpha,\Delta)=[b_s']$.
\end{proof}

%


\subsection{$P_{s+1}(\alpha,\Delta)=[b,b,\ldots,b]$ where $b_s'=mb$}\label{subsec Ps+1alphaDeltaismb}

Let $b_s'=mb$ where $m\geqslant 1$ and $b\geqslant 1$, and suppose
\begin{center}$P_{s+1}(\alpha,\Delta)=[p_1,p_2,\ldots,p_m]$ where $p_1=p_2=\cdots =p_m=b$.\end{center} Let $j\in X$ and let $\gamma\in B_j$. Then $|B_j^\gamma|=m$. Let $S\in B_j^\gamma\subseteq \mathcal{B}(c_s',j)$. Then $|\Gamma_1(\gamma)\cap \Delta(S)|=b$ and $\gamma\in B_j(S)$. Hence $\gamma\in \bigcap\limits_{S\in B_j^\gamma}B_j(S)\neq\emptyset$. Let $S_1,S_2,\ldots,S_m\in \mathcal{B}(c_s',j)$ be pairwise distinct. Define \begin{center}$B_j(S_1,S_2,\ldots,S_m)=\bigcap\limits_{r=1}^m B_j(S_r)$,\end{center}
\begin{center}$\Delta_j(S_1,S_2,\ldots,S_m)=\bigcup\limits_{\gamma\in B_j(S_1,S_2,\ldots,S_m)} \Gamma_1(\gamma)\cap \Delta_j\subseteq \Delta_j$\end{center} and
\begin{center}$\mathcal{B}_m(c_s',j)=\bigl\{\{S_1,S_2,\ldots,S_m\}\mid B_j(S_1,S_2,\ldots,S_m)\neq\emptyset\text{ where }\newline S_1,S_2,\ldots,S_m\in \mathcal{B}(c_s',j)\text{ are pairwise distinct}\bigr\}$.\end{center} So $B_j(S_1,S_2,\ldots,S_m)\subseteq B_j(S_r)\subseteq B_j$ where $1\leqslant r\leqslant m$, and
\begin{center}$|\mathcal{B}_m(c_s',j)|\leqslant \binom{|\mathcal{B}(c_s',j)|}{m}$.\end{center} Suppose $B_j^\gamma=\{S_1,S_2,\ldots,S_m\}$, then by definition $\gamma\in B_j(S_1,S_2,\ldots,S_m)\neq\emptyset$. So we have $B_j^\gamma\in \mathcal{B}_m(c_s',j)$. For any $\tilde{T}=\{T_1,T_2,\ldots,T_m\}\in \mathcal{B}_m(c_s',j)$, we have $B_j(\tilde{T})=B_j(T_1,T_2,\ldots,T_m)\neq\emptyset$. Take $\gamma\in B_j(\tilde{T})$. Then for $1\leqslant r\leqslant m$, we have $\gamma\in B_j(T_r)$ i.e. $\Gamma_1(\gamma)\cap \Delta(T_r)\neq\emptyset$, and so $|\Gamma_1(\gamma)\cap \Delta(T_r)|=b$. Thus $B_j^\gamma=\{T_1,T_2,\ldots,T_m\}$. Thus we have
\begin{center}$\mathcal{B}_m(c_s',j)=\{B_j^\gamma\mid \gamma\in B_j\}$.\end{center}
If $\{S_1,S_2,\ldots,S_m\}$ and $\{T_1,T_2,\ldots,T_m\}$ are distinct in $\mathcal{B}_m(c_s',j)$, then
\begin{center}$B_j(S_1,S_2,\ldots,S_m)\cap B_j(T_1,T_2\ldots,T_m)=\emptyset$.\end{center} If $S\in \mathcal{B}(c_s',j)$, then
\begin{equation}\label{eqn BjS} B_j(S)=\bigcup\limits_{\{S,T_1,\ldots,T_{m-1}\}\in \mathcal{B}_m(c_s',j)}B_j(S,T_1,\ldots,T_{m-1})\end{equation} and
\begin{center}$B_j=\bigcup\limits_{\{S_1,S_2,\ldots,S_m\}\in \mathcal{B}_m(c_s',j)}B_j(S_1,S_2,\ldots,S_m)$.\end{center} Let $S\in \mathcal{B}(c_s',j)$ and
\begin{center}$\mathcal{B}_S(c_s',j)=\bigl\{\{S,T_1,T_2,\ldots,T_{m-1}\}\in \mathcal{B}_m(c_s',j)\bigr\}$.\end{center} Thus
\begin{center}$B_j(S)^{m,\mathcal{P}(c_s')}=\bigl\{B_j(S,T_1,T_2,\ldots,T_{m-1})\mid\{S,T_1,T_2,\ldots,T_{m-1}\}\in \mathcal{B}_S(c_s',j)\bigr\}$\end{center} forms a partition of $B_j(S)$ and
\begin{center}$B_j^{m,\mathcal{P}(c_s')}=\bigl\{B_j(S_1,S_2,\ldots,S_m)\mid\{S_1,S_2,\ldots,S_m\}\in \mathcal{B}_m(c_s',j)\bigr\}$\end{center} forms a partition of $B_j$. Let $\{S,T_1,T_2,\ldots,T_{m-1}\}\in \mathcal{B}_m(c_s',j)$, we define
\begin{center}$\Delta_S(S,T_1,T_2,\ldots,T_{m-1})=\bigcup\limits_{\gamma\in B_j(S,T_1,T_2,\ldots,T_{m-1})}\Gamma_1(\gamma)\cap \Delta(S)\subseteq \Delta(S)$.\end{center} Then $\Delta_S(S,T_1,T_2,\ldots,T_{m-1})\subseteq \Delta_j(S,T_1,T_2,\ldots,T_{m-1})\subseteq \Delta_j$,
\begin{center}$\Delta_j(S_1,S_2,\ldots,S_m)=\bigcup\limits_{r=1}^m\Delta_{S_r}(S_1,S_2,\ldots,S_m)$,\end{center}
\begin{center}$\Delta(S)=\bigcup\limits_{\{S,T_1,\ldots,T_{m-1}\}\in \mathcal{B}_m(c_s',j)}\Delta_S(S,T_1,\ldots,T_{m-1})$\end{center} and \begin{center}$\Delta_j=\bigcup\limits_{\{S_1,S_2,\ldots,S_m\}\in \mathcal{B}_m(c_s',j)}\Delta_j(S_1,S_2,\ldots,S_m)$.\end{center}
Thus
\begin{center}$\mathcal{P}_S(c_s',j)=\bigl\{\Delta_S(S,T_1,T_2,\ldots,T_{m-1})\mid\{S,T_1,T_2,\ldots,T_{m-1}\}\in \mathcal{B}_S(c_s',j)\bigr\}$\end{center} forms a partition of $\Delta(S)$ and
\begin{center}$\mathcal{P}_m(c_s',j)=\bigl\{\Delta_j(S_1,S_2,\ldots,S_m)\mid\{S_1,S_2,\ldots,S_m\}\in \mathcal{B}_m(c_s',j)\bigr\}$\end{center} forms a partition of $\Delta_j$.

\begin{lemma}\label{lem structural information P_s+1=[bb...b] where bs'=mb}
Let $b_s'=mb$ where $m\geqslant 1$ and $b\geqslant 1$, and suppose $P_{s+1}(\alpha,\Delta)=[p_1,p_2,\ldots,p_m]$ where $p_1=p_2=\cdots =p_m=b$. Let $j\in X$ and let $S\in \mathcal{B}(c_s',j)$.
\begin{enumerate}
\item The edges between $B_j$ and $\Delta_j$ is a disjoint union of $K_{1,b_s'}$'s.
\item Let $\{S_1,S_2,\ldots,S_m\}\in \mathcal{B}_m(c_s',j)$. Then the edges between \newline $B_j(S_1,S_2,\ldots,S_m)$ and $\Delta_j(S_1,S_2,\ldots,S_m)$ is a disjoint union of $K_{1,b_s'}$'s.
\item Let $\{S,T_1,T_2,\ldots,T_{m-1}\}\in \mathcal{B}_m(c_s',j)$. Then the edges between \newline $B_j(S,T_1,T_2,\ldots,T_{m-1})$ and $\Delta_S(S,T_1,T_2,\ldots,T_{m-1})$ is a disjoint union of $K_{1,b}$'s.
\item $B_j(S)^{m,\mathcal{P}(c_s')}$ is a partition of $B_j(S)$ and $B_j^{m,\mathcal{P}(c_s')}$ is a partition of $B_j$.
\item $\mathcal{P}_S(c_s',j)$ is a partition of $\Delta(S)$ and $\mathcal{P}_m(c_s',j)$ is a partition of $\Delta_j$.
\end{enumerate}
\end{lemma}

\begin{figure}[htb]
\centering
\includegraphics[width=0.70\textwidth]{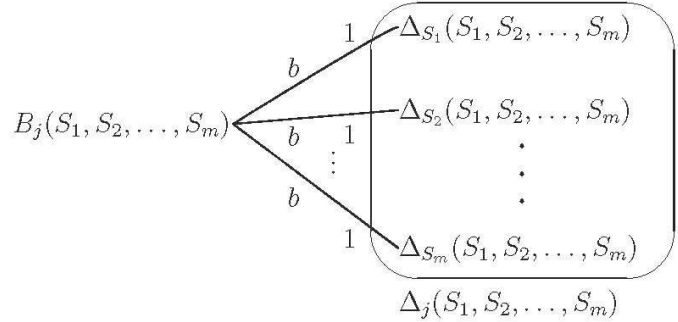}
\caption{The adjacency relations and intersection numbers between $B_j(S_1,S_2,\ldots,S_m)$ and $\Delta_j(S_1,S_2,\ldots,S_m)$.}
\label{fig: partitionofDeltajSonetom}
\end{figure}

%
%


By Lemma~\ref{lem Delta j gamma 1 to 1 Delta h varepsilon}, the stabilizer $G_{\alpha\beta_j}$ induces a transitive group on $\mathcal{B}_m(c_s',j)$, and so $G_{\alpha\beta_j}$ induces a transitive group on $\mathcal{B}_j^{m,\mathcal{P}(c_s')}$. Then \begin{equation}\label{eqn fmj}\f_j(m)=|B_j(S_1,S_2,\ldots,S_m)|\end{equation} is independent of $\{S_1,S_2,\ldots,S_m\}\in \mathcal{B}_m(c_s',j)$. Let $S\in \mathcal{B}(c_s',j)$. We define $\mathcal{B}_m^S(c_s',j)=\Bigl\{\{S,T_1,T_2,\ldots,T_{m-1}\}\in \mathcal{B}_m(c_s',j)\Bigr\}$.

Recall that
\begin{eqnarray*}
\mathcal{B}(c_s',j)&=&\{S\in \mathcal{B}(c_s')\mid j\in S\},\\
\mathcal{P}(c_s',j)&=&\{\Delta(S)\mid S\in \mathcal{B}(c_s',j)\},\\
B_j^{\mathcal{P}(c_s')}&=&\{B_j(S)\mid S\in \mathcal{B}(c_s',j)\}.
\end{eqnarray*}
By Lemma~\ref{lem DeltaiS DeltaiT}, the stabilizer $G_{\alpha\beta_j}$ induces a transitive group on $\mathcal{P}(c_s',j)$, and so $G_{\alpha\beta_j}$ induces a transitive group on $B_j^{\mathcal{P}(c_s')}$. Then \begin{equation}\label{eqn fjS}\f_j=|B_j(S)|\end{equation} is independent of $S\in \mathcal{B}(c_s',j)$.

\begin{lemma}\label{lem Bm(cs',j) is a 1-design}
Let $j\in X$ and $\mu=|\mathcal{B}(c_s',j)|$. Then $\lambda=|\mathcal{B}_m^S(c_s',j)|$ is independent of $S\in \mathcal{B}(c_s',j)$, and so $(\mathcal{B}(c_s',j),\mathcal{B}_m(c_s',j))$ is a $1-(\mu,m,\lambda)$ design.
\end{lemma}

\begin{proof}
By Equation~(\ref{eqn BjS}), we know that $\f_j=|B_j(S)|=|\mathcal{B}_m^S(c_s',j)|\f_j(m)$. So $|\mathcal{B}_m^S(c_s',j)|$ is independent of $S$.
\end{proof}

Now we give an example of Lemma~\ref{lem structural information P_s+1=[bb...b] where bs'=mb} and Lemma~\ref{lem Bm(cs',j) is a 1-design}. Let $k=4$, $X=\{1,2,3,4\}$, $c_s'=2$, $\mathcal{B}(c_s')=\{\{1,2\},\{1,3\},\{1,4\},\{2,3\},\{2,4\},\{3,4\}\}$, $m=2$, $b=1$ and $P_{s+1}(\alpha,\Delta)=[1,1]$. We first consider the adjacency relations between $B_1$ and $\Delta_1$. Then $\mathcal{B}(c_s',1)=\{\{1,2\},\{1,3\},\{1,4\}\}$. For $S=\{j,h\}\in \mathcal{B}(c_s')$, we use $j_h$ or $h_j$ to denote the set $S$, i.e. \begin{center}$j_h=h_j=\{j,h\}$.\end{center}
Then $\mathcal{B}(c_s',1)=\{2_1,3_1,4_1\}$. By Table~\ref{table 1-design for valency 3}, there is only one $1$-$(3,2,x)$ design on three points. Hence \begin{equation*}
\mathcal{B}_m(c_s',1)=\Bigl\{ \{2_1,3_1\}, \{2_1,4_1\}, \{3_1,4_1\}   \Bigr\}.
\end{equation*}
Note that
\begin{eqnarray*}
B_1(2_1)&=&B_1(2_1,3_1)\cup B_1(2_1,4_1),\\
B_1(3_1)&=&B_1(2_1,3_1)\cup B_1(3_1,4_1),\\
B_1(4_1)&=&B_1(2_1,4_1)\cup B_1(3_1,4_1);
\end{eqnarray*}
and
\begin{eqnarray*}
\Delta(2_1)&=&\Delta_{2_1}(2_1,3_1)\cup \Delta_{2_1}(2_1,4_1)=\Delta(1_2)=\Delta_{1_2}(1_2,3_2)\cup \Delta_{1_2}(1_2,4_2),\\
\Delta(3_1)&=&\Delta_{3_1}(2_1,3_1)\cup \Delta_{3_1}(3_1,4_1)=\Delta(1_3)=\Delta_{1_3}(1_3,2_3)\cup \Delta_{1_3}(1_3,4_3),\\
\Delta(4_1)&=&\Delta_{4_1}(2_1,4_1)\cup \Delta_{4_1}(3_1,4_1)=\Delta(1_4)=\Delta_{1_4}(1_4,2_4)\cup \Delta_{1_4}(1_4,3_4).
\end{eqnarray*}
The sets $B_1(2_1,3_1)$, $B_1(2_1,4_1)$ and $B_1(3_1,4_1)$ form a partition of $B_1$, and the sets $\Delta_{2_1}(2_1,3_1)$, $\Delta_{2_1}(2_1,4_1)$, $\Delta_{3_1}(2_1,3_1)$, $\Delta_{3_1}(3_1,4_1)$, $\Delta_{4_1}(2_1,4_1)$, $\Delta_{4_1}(3_1,4_1)$ form a partition of $\Delta_1=\Delta(2_1)\cup \Delta(3_1)\cup \Delta(4_1)$. Let $S,T\in \mathcal{B}(c_s',1)$. Since $b=1$, the edges between $B_1(S,T)$ and $\Delta_S(S,T)$ is a matching. So we get the adjacency relations between $B_1$ and $\Delta_1$ (see Figure~\ref{fig: B1}). The adjacency relations between $B_j$ and $\Delta_j$ for $2\leqslant j\leqslant 4$ are similarly got (see Figure~\ref{fig: adjacencyrelationskfourPoneone}). The edges in Figure~\ref{fig: adjacencyrelationskfourPoneone} represents the matching between $B_j(S,T)$ and $\Delta_S(S,T)$ for $S,T\in \mathcal{B}(c_s',j)$ and $1\leqslant j\leqslant 4$.

\begin{figure}[htb]
\centering
\subfigure[$j=1$]{
\label{fig: B1}
\includegraphics[width=0.48\textwidth]{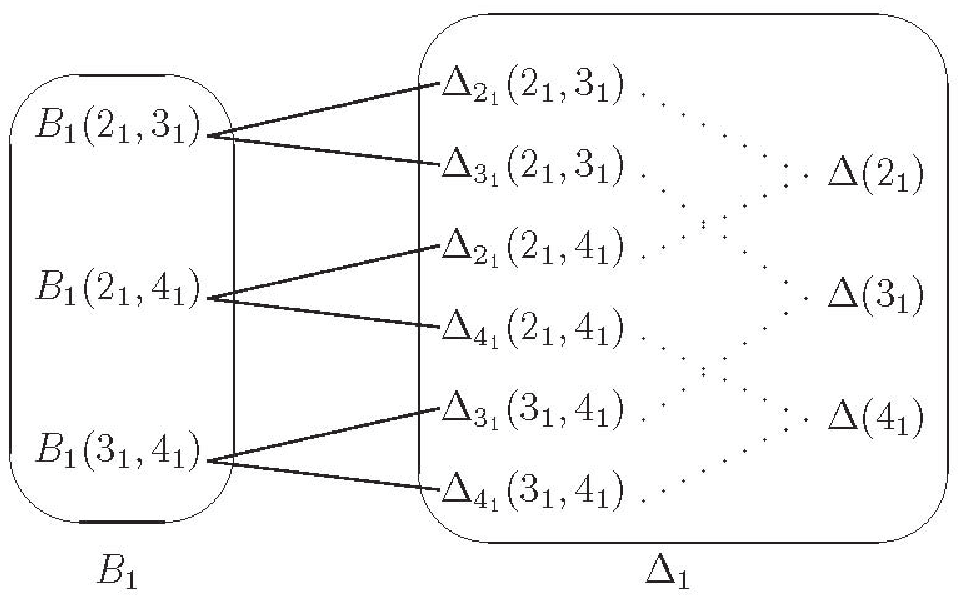}}
\subfigure[$j=2$]{
\label{fig: B2}
\includegraphics[width=0.48\textwidth]{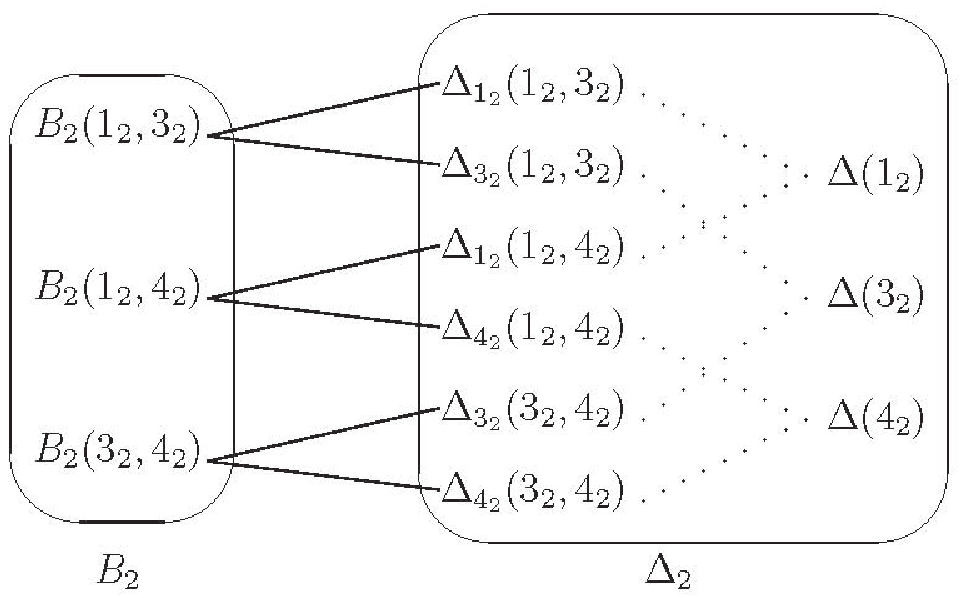}}
\subfigure[$j=3$]{
\label{fig: B3}
\includegraphics[width=0.48\textwidth]{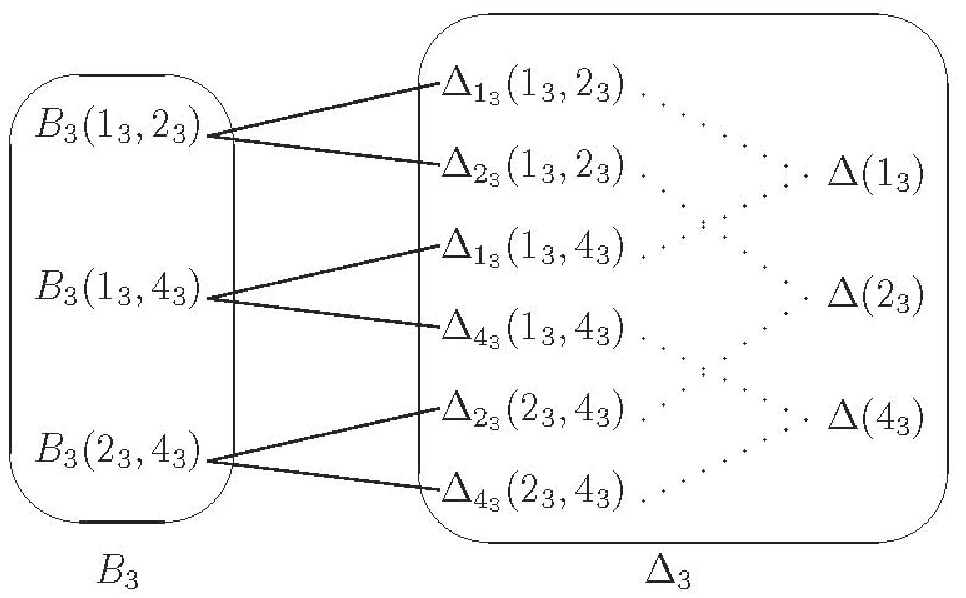}}
\subfigure[$j=4$]{
\label{fig: B4}
\includegraphics[width=0.48\textwidth]{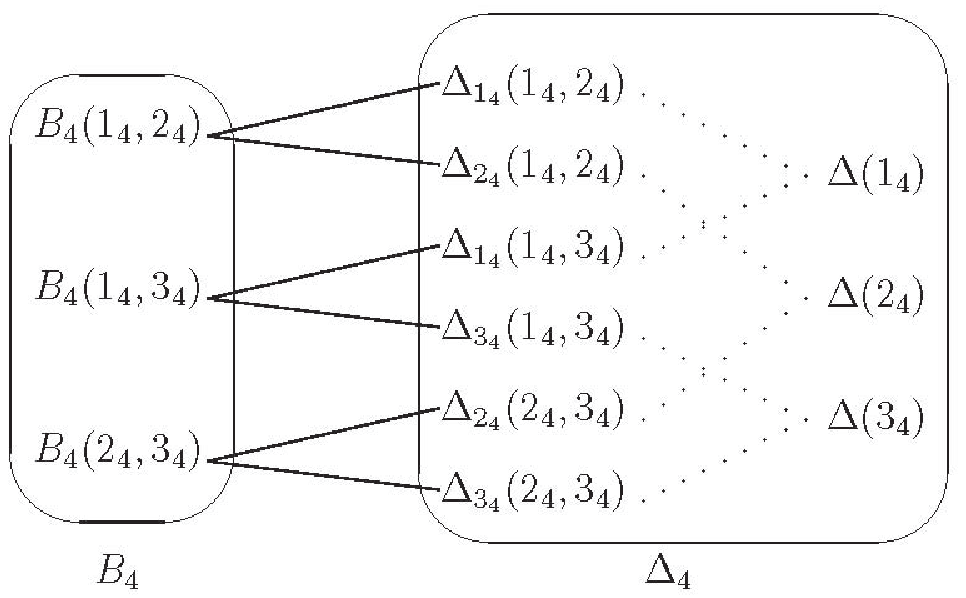}}
\caption{The adjacency relations between $B_j$ and $\Delta_j$ for $1\leqslant j\leqslant 4$.}
\label{fig: adjacencyrelationskfourPoneone}
\end{figure}

Since $G_{\alpha\beta_j}$ induces a transitive group on $B_j^{\mathcal{P}(c_s')}$, there exists an integer $\tau\geqslant 1$ such that the induced action of $G_{\alpha\beta_j}$ on $B_j^{\mathcal{P}(c_s')}$ is $\tau$-homogeneous. If $\tau\geqslant m$, then $\mathcal{B}_m(c_s',j)$ consists of all $m$-subsets of $\mathcal{B}(c_s',j)$. In the following, we may assume $\tau<m$. Let $\theta\geqslant 1$ such that $\tau+\theta=m$. Let $S_1,S_2,\ldots,S_\tau\in \mathcal{B}(c_s',j)$ be pairwise distinct. We define
\begin{eqnarray*}
\mathcal{B}_m^{S_1,S_2,\ldots,S_\tau}(c_s',j)&=&\Bigl\{\{S_1,S_2,\ldots,S_\tau,T_1,T_2,\ldots,T_\theta\}\in \mathcal{B}_m(c_s',j)\Bigr\},\\
B_j(S_1,S_2,\ldots,S_\tau)&=&\bigcap\limits_{r=1}^\tau B_j(S_r).
\end{eqnarray*}
Then \begin{equation}\label{eqn BjS1S2...Stau}B_j(S_1,S_2,\ldots,S_\tau)=\bigcap\limits_{\{S_1,S_2,\ldots,S_m\}\in \mathcal{B}_m^{S_1,S_2,\ldots,S_\tau}(c_s',j)}B_j(S_1,S_2,\ldots,S_m)\end{equation} and
\begin{equation}\label{eqn fjtau}\f_j(\tau)=|B_j(S_1,S_2,\ldots,S_\tau)|\end{equation} is independent of $S_1,S_2,\ldots,S_\tau$.

\begin{lemma}\label{lem Bm(cs',j) is a tau-design}
Let $j\in X$ and $\mu=|\mathcal{B}(c_s',j)|$. Let $S_1,S_2,\ldots,S_\tau\in \mathcal{B}(c_s',j)$ be pairwise distinct. Then the size $\lambda_\tau=|\mathcal{B}_m^{S_1,S_2,\ldots,S_\tau}(c_s',j)|$ is independent of $S_1,S_2,\ldots,S_\tau$, and so $(\mathcal{B}(c_s',j),\mathcal{B}_m(c_s',j))$ is a $\tau-(\mu,m,\lambda_\tau)$ design.
\end{lemma}

\begin{proof}
By Equation~(\ref{eqn BjS1S2...Stau}), we know that \begin{center}$\f_j(\tau)=|B_j(S)|=|\mathcal{B}_m^{S_1,S_2,\ldots,S_\tau}(c_s',j)|\f_j(m)$.\end{center} So $|\mathcal{B}_m^{S_1,S_2,\ldots,S_\tau}(c_s',j)|$ is independent of $S_1,S_2,\ldots,S_\tau$.
\end{proof}

\section{Cubic $s$-Distance-transitive graphs}\label{sec cubic s-Distance-transitive graphs}

%
%

\begin{lemma}\label{lem 3ndt is DR}
Let $\Gamma$ be a cubic graph with diameter $d\geqslant 2$. Let $s=d-1$ and $G\leqslant \Aut(\Gamma)$. If $\Gamma$ is $(G,s)$-distance-transitive, then $\Gamma$ is distance-regular.
\end{lemma}

\begin{proof}
Suppose $\Gamma$ is $(G,s)$-distance-transitive. There are two cases we need to consider. If $\Gamma$ is $G$-distance-transitive, then $\Gamma$ is distance-regular. Now we assume that $\Gamma$ is not $G$-distance-transitive. Take any vertex $\alpha\in V(\Gamma)$. Then the vertex stabilizer $G_\alpha$ is transitive on $\Gamma_i(\alpha)$ for $1\leqslant i\leqslant s$ but not on $\Gamma_d(\alpha)$. By Corollary~\ref{cor s+1 orbit}, there has exactly two $G_\alpha$-orbits $\Delta_1$ and $\Delta_2$ in $\Gamma_d(\alpha)$ and $b_s(\Gamma,\alpha)=2$. By Lemma~\ref{lem s-dt properties of intersection numbers}, we have $(c_i(\Gamma,\alpha),a_i(\Gamma,\alpha),b_i(\Gamma,\alpha))=(1,0,2)$ for $1\leqslant i\leqslant s$. This implies that $\Gamma$ is $(G,s)$-arc-transitive and the girth of $\Gamma$ is at least $2d$. Let $\delta_i=|\Delta_i|$, $b_i'=b_s(\Gamma,G_\alpha,\Delta_i)$, $c_i'=c_s'(\Gamma,G_\alpha,\Delta_i)$, and $a_{i,j}=\kappa(\Gamma,G_\alpha;\Delta_i,\Delta_j)$ for $1\leqslant i,j\leqslant 2$. In the following, we will show $c_1'=c_2'$ which implies that $\Gamma$ is distance-regular. We prove this by contrary.

From now on, we suppose $c_1'\neq c_2'$. Without loss of generality, we may assume $c_1' < c_2'$. We first show that the diameter $d\geqslant 3$. We need consider three cases $(c_1',c_2')=(1,2),(1,3)$ or $(2,3)$. We first suppose $c_i'=2$ for some $1\leqslant i\leqslant 2$. Then $\delta_i=k_{d-1}/2=|\Gamma_{d-1}(\alpha)|/2=3\cdot 2^{d-3}$. Since $\delta_i$ is an integer, we have $d\geqslant 3$. Now let $(c_1',c_2')=(1,3)$. Note that $k_s=|\Gamma_s(\alpha)|=3\cdot 2^{d-2}$. Then $\delta_1=k_s=3\cdot 2^{d-2}$, $\delta_2=k_s/3=2^{d-2}$ and $3a_{1,2}=a_{2,1}$. By valency restriction, we have $a_{2,2}=0$ and $a_{1,2}=a_{2,1}=0$. So $a_{1,1}=2$. If $d=2$, then $\delta_1=3$ and the induced subgraph on $\Delta_1$ is a triangle. This contradicts the girth of $\Gamma$ is at least $4$. So $d\geqslant 3$.

Now $\Gamma$ is $(G,2)$-arc transitive, and so $G_\alpha$ is $2$-transitive on $\Gamma_1(\alpha)$ which implies $G_\alpha$ is $2$-homogeneous on $\Gamma_1(\alpha)$. By Theorem~\ref{theorem t-design}, for $1\leqslant i\leqslant 2$, there is a $2$-$(3,c_i',x_i)$ design $D(\Delta_i)$ relative to the $G_\alpha$-orbit $\Delta_i$. There are exactly two $2$-designs on $3$ points (see Table~\ref{table 1-design for valency 3}). Hence $(c_1',c_2')=(2,3)$. Then $\delta_1=k_s/2=3\cdot 2^{d-3}$, $\delta_2=k_s/3=2^{d-2}$, $3a_{1,2}=2a_{2,1}$. By valency restriction, we have $a_{2,2}=0$, $a_{1,2}=a_{2,1}=0$. So $a_{1,1}=1$.
We use notations as in Section~\ref{sec Local actions and designs}, and we suggest the reader to draw a picture as in Figure~\ref{fig: Adjacency relations k3bi'1}. Take $\beta=\beta_3$. Then $\Delta_2=(\Delta_{2})_3\subseteq \Gamma_s(\beta)$, $(\Delta_{1})_3=\Delta_1(\{1,3\})\cup \Delta_1(\{2,3\})\subseteq \Gamma_s(\beta)$, and the sets $B_1$, $B_2$ and $(\Delta_{1})_3'=\Delta_1(\{1,2\})$ form a partition of $\Gamma_d(\beta)$. Then for any vertex in $(\Delta_{1})_3'$, it has exactly two neighbors in $B_1\cup B_2$. Thus any vertex in $(\Delta_{1})_3'$ has degree at least two in the induced subgraph on $\Gamma_d(\beta)$. While the induced subgraph on $\Gamma_d(\alpha)$ contains no vertex of degree greater than one. A contradiction.
\end{proof}

Cubic distance-regular graphs were classified in~\cite{Cubic distance-regular graphs}, and these graphs are either distance-transitive~\cite{Biggs and Smith On trivalent graphs} or the Tutte's $12$-cage which is not vertex transitive. So a vertex-transitive and distance-regular graph with valency three must be distance-transitive.

\begin{corollary}\label{cor 3ndt is DT}
Let $\Gamma$ be a cubic graph with diameter $d\geqslant 2$. Let $s=d-1$ and $G\leqslant \Aut(\Gamma)$. If $\Gamma$ is $(G,s)$-distance-transitive, then $\Gamma$ is distance-transitive.
\end{corollary}

Magma codes of cubic distance-transitive graphs can be found on Marston Conder's homepage~\cite{MarstonHomepage}. Let $\Gamma$ be a cubic distance-transitive graph with diameter $d\geqslant 2$ such that there is an automorphism group $G\leqslant \Aut(\Gamma)$ satisfying $\Gamma$ is $(G,s)$-distance-transitive but not $G$-distance-transitive where $s=d-1$. By Magma~\cite{Magma}, we can check that the only possible of such a cubic distance-transitive graph is $\Gamma=K_{3,3}$ with automorphism group $G=C_3\wr C_2<\Aut(\Gamma)=\S_3\wr C_2$. Here $C_n$ is a cyclic group of order $n$, $\S_n$ is the symmetric group on $n$ letters and the symbol $\wr$ means the wreath product. By the proof of Lemma~\ref{lem 3ndt is DR} and this discussion, we have the following result.


\begin{theorem}\label{thm 3ndt girth at most 2d-1}
Let $\Gamma$ be a cubic graph with diameter $d\geqslant 3$ and girth $g$. Let $s=d-1$ and let $G\leqslant \Aut(\Gamma)$ such that $\Gamma$ is $(G,s)$-distance-transitive but not $G$-distance-transitive. Then $g\leqslant 2d-1$.
\end{theorem}

\begin{proof}
Suppose $g\geqslant 2d$. By the proof of Lemma~\ref{lem 3ndt is DR}, the graph $\Gamma$ is distance-regular and so distance-transitive. The only possible graph is $K_{3,3}$ with diameter $d=2<3$. This is a contradiction.
\end{proof}

\section{Tetravalent $s$-Distance-transitive graphs\newline with large girth}\label{sec tetravalent s-Distance-transitive graphs}


\begin{lemma}\label{lem 4ndt girth at least 2d is DR}
Let $\Gamma$ be a tetravalent graph with diameter $d\geqslant 3$ and girth $g\geqslant 2d$. Let $s=d-1$ and let $G\leqslant \Aut(\Gamma)$ such that $\Gamma$ is $(G,s)$-distance-transitive. Then $\Gamma$ is distance-regular.
\end{lemma}

\begin{proof}
If $\Gamma$ is $G$-distance-transitive, then $\Gamma$ is distance-regular. Suppose $\Gamma$ is not $G$-distance-transitive. Take any vertex $\alpha\in V(\Gamma)$. Then the vertex stabilizer $G_\alpha$ is transitive on $\Gamma_i(\alpha)$ for $1\leqslant i\leqslant s$ but not on $\Gamma_d(\alpha)$. Let $\Delta_1$, $\Delta_2$, $\ldots$, $\Delta_n$ be all the $G_\alpha$-orbits in $\Gamma_d(\alpha)$. By Corollary~\ref{cor s+1 orbit}, we have $b_s(\Gamma,\alpha)=3$ and $2\leqslant n\leqslant 3$. Since $g\geqslant 2d=2s+2$, by Lemma~\ref{lem alpha girth >=2s+2}, we have $(c_i(\Gamma,\alpha),a_i(\Gamma,\alpha),b_i(\Gamma,\alpha))=(1,0,3)$ for $1\leqslant i\leqslant s$. This implies that $\Gamma$ is $(G,s)$-arc-transitive, and so $\Gamma$ is $(G,2)$-arc-transitive since $s\geqslant 2$. Hence $G_\alpha$ is $2$-transitive on $\Gamma(\alpha)$ which implies $G_\alpha$ is $2$-homogeneous on $\Gamma(\alpha)$. By Theorem~\ref{theorem t-design}, for $1\leqslant i\leqslant n$, there is a $2$-$(4,c_i',x_i)$ design $D(\Delta_i)$ relative to the $G_\alpha$-orbit $\Delta_i$. There are exactly three $2$-designs on $4$ points (see Table~\ref{table 1-design for valency 4}): the $2$-$(4,2,1)$ design, the $3$-$(4,3,1)$ design and the $4$-$(4,4,1)$ design. Then $2\leqslant c_i'\leqslant 4$ for $1\leqslant i\leqslant n$. Let $\delta_i=|\Delta_i|$, $b_i'=b_s'(\Gamma,G_\alpha,\Delta_i)$, $c_i'=c_s'(\Gamma,G_\alpha,\Delta_i)$, and $a_{i,j}=\kappa(\Gamma,G_\alpha;\Delta_i,\Delta_j)$ for $1\leqslant i,j\leqslant n$.

Let $\kappa_s=|\Gamma_s(\alpha)|=4\cdot 3^{d-2}$ and let $1\leqslant i,j\leqslant n$ with $i\neq j$. Then by counting edges between $\Gamma_s(\alpha)$ and $\Delta_i$ we have $\kappa_s b_i'=\delta_i c_i'$, and by counting edges between $\Delta_i$ and $\Delta_j$ we have $\delta_i a_{i,j}=\delta_j a_{j,i}$.

In the following, we will show $c_1'=c_2'=\cdots=c_n'$ which implies that $\Gamma$ is distance-regular. We prove this by contrary. From now on, we suppose $|\{c_1',c_2',\ldots,c_n'\}|>1$. We use notations as in Section~\ref{sec Local actions and designs} and Section~\ref{sec Adjacency relations}, and we suggest the reader to draw pictures as in Figure~\ref{fig: Adjacency relations k4bi'1}.

Take $\beta=\beta_1$. Note that $\Gamma_d(\alpha)=\Delta_1\cup\Delta_2\cup\Delta_3$. So for any vertex $x\in \Gamma_d(\alpha)$, the number of neighbors of $x$ in $\Gamma_d(\alpha)$ is at most two, i.e. the induced degree $\deg_{[\Gamma_d(\alpha)]}(x)\leqslant 2$. So the induced degree on $[\Gamma_d(\beta)]$ is also $\leqslant 2$. Suppose for some $1\leqslant i\leqslant 3$, the intersection number $c_i'=3$. Then $B_2\cup B_3\cup B_4\cup (\Delta_i)_1'\subseteq \Gamma_d(\beta)$ where $(\Delta_i)_1'=\Delta_i\setminus (\Delta_i)_1=\Delta_i(\{2,3,4\})$. Let $y\in (\Delta_i)_1'$. Then for $2\leqslant j\leqslant 4$, the vertex $y$ has one neighbour in $B_j$. Thus $\deg_{[\Gamma_d(\beta)]}(y)\geqslant 3$, a contradiction. Hence $c_i'\in \{2,4\}$ for $1\leqslant i\leqslant n$.

\begin{case}\label{case n=3}The case $n=3$.\end{case}

\begin{figure}[htb]
\centering
\includegraphics[width=0.70\textwidth]{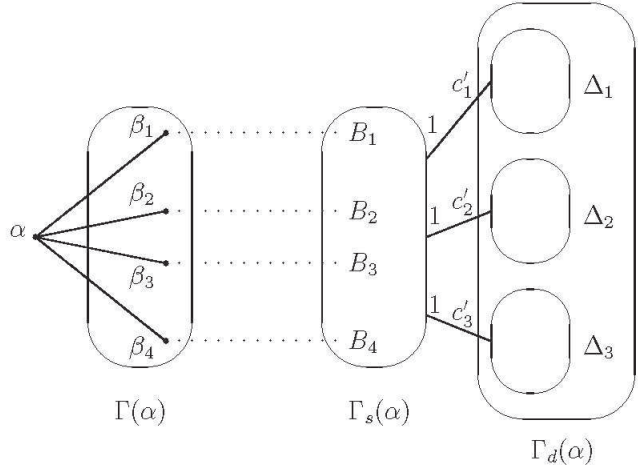}
\caption{Three $G_\alpha$-orbits $\Delta_1$, $\Delta_2$ and $\Delta_3$ in $\Gamma_d(\alpha)$.}
\label{fig: threeorbitsvalencyfour}
\end{figure}

Suppose $n=3$. Then $b_1'=b_2'=b_3'=1$ (see Figure~\ref{fig: threeorbitsvalencyfour}). Without loss of generality, we may assume $c_1'\leqslant c_2'\leqslant c_3'$. We only need to consider the following two cases $(c_1',c_2',c_3')=(2,4,4)$ or $(2,2,4)$.

Let $(c_1',c_2',c_3')=(2,4,4)$. Then $a_{3,1}=a_{2,1}=a_{1,2}=a_{1,3}=a_{2,3}=a_{3,2}=a_{2,2}=a_{3,3}=0$ and $a_{1,1}=2$. For any vertex $u\in V\Gamma$, let \begin{equation*}\partial_\Gamma^u(2)=\{x\in \Gamma_d(u)\mid \deg_{[\Gamma_d(u)]}(x)=2\}.\end{equation*} Since $\Gamma$ is vertex-transitive, the size $\partial_\Gamma(2)=|\partial_\Gamma^u(2)|$ is independent of $u$. We know $\partial_\Gamma^\alpha(2)=\Delta_1$. Now
\begin{eqnarray*}
\Gamma_d(\beta)&=&B_2\cup B_3\cup B_4\cup (\Delta_1)_1',\\
(\Delta_1)_1'&=&\Delta_1\setminus (\Delta_1)_1=\Delta_1(\{2,3\})\cup \Delta_1(\{2,4\})\cup \Delta_1(\{3,4\}),\\
(\Delta_1)_1&=&\Delta_1(\{1,2\})\cup \Delta_1(\{1,3\})\cup \Delta_1(\{1,4\}).
\end{eqnarray*}
Then $\partial_\Gamma^\beta(2)=(\Delta_1)_1'\subsetneqq \Delta_1$, which implies $\partial_\Gamma(2)=|\partial_\Gamma^\beta(2)|<|\partial_\Gamma^\alpha(2)|=\partial_\Gamma(2)$. This is a contradiction.

Let $(c_1',c_2',c_3')=(2,2,4)$. Then \begin{equation*}|V\Gamma|=11\cdot 3^{d-2}-1.\end{equation*} This case is left in Case~\ref{case rest cases}.

\begin{case}\label{case n=2}The case $n=2$.\end{case}

\begin{figure}[htb]
\centering
\includegraphics[width=0.80\textwidth]{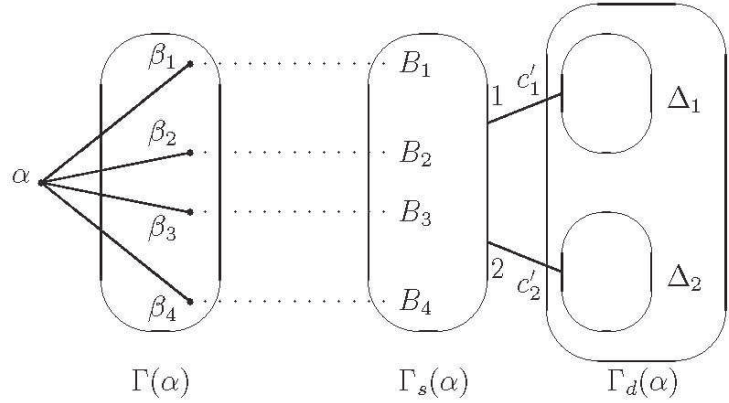}
\caption{Two $G_\alpha$-orbits $\Delta_1$ and $\Delta_2$ in $\Gamma_d(\alpha)$.}
\label{fig: twoorbitsvalencyfour}
\end{figure}

%
%
%

Suppose $n=2$. Then $\{b_1',b_2'\}=\{1,2\}$. Without loss of generality, we may assume $b_1'=1$ and $b_2'=2$ (see Figure~\ref{fig: twoorbitsvalencyfour}). We only need to consider the following two cases $(c_1',c_2')=(2,4)$ or $(4,2)$.

Let $(c_1',c_2')=(2,4)$. Then $a_{2,2}=a_{2,1}=a_{1,2}=0$ and $a_{1,1}=2$. For any vertex $x\in \Gamma_d(\alpha)$, the induced degree $\deg_{[\Gamma_d(\alpha)]}(x)=0$ or $2$. So the induced degree on $[\Gamma_d(\beta)]$ is also $0$ or $2$. Now
\begin{eqnarray*}
\Gamma_d(\beta)&=&B_2\cup B_3\cup B_4\cup (\Delta_1)_1',\\
(\Delta_1)_1'&=&\Delta_1\setminus (\Delta_1)_1=\Delta_1(\{2,3\})\cup \Delta_1(\{2,4\})\cup \Delta_1(\{3,4\}).
\end{eqnarray*}
Let $S\in \{\{2,3\},\{2,4\},\{3,4\}\}$ and let $j\in S$ with $2\leqslant j\leqslant 4$. Take a vertex $y\in B_j(S)_{\Delta_1}$. Then $y$ has only one neighbor in $\Gamma_d(\beta)$ (more precisely, the neighbor is in $\Delta_1(S)$). So $\deg_{[\Gamma_d(\beta)]}(y)=1$, a contradiction.

Let $(c_1',c_2')=(4,2)$. Then \begin{center}$|V\Gamma|=11\cdot 3^{d-2}-1$.\end{center} This case is left in Case~\ref{case rest cases}.


\begin{case}\label{case rest cases}The case $|V\Gamma|=11\cdot 3^{d-2}-1$.\end{case}

\begin{table}[!ht]
\centering
    \begin{tabular}{|c|c|c|c|c|}
        \hline
        $d$ & $3$ & $4$ & $5$ & $8$ \\
        \hline
        $|V\Gamma|$ & $32$ & $98$ & $296$ & $8018$ \\
        \hline
    \end{tabular}
    \caption{The diameter $d$ and the order of the graph $\Gamma$}\label{table diameter and order tetravalent 2AT}
\end{table}

The graph $\Gamma$ is $(G,s)$-arc-transitive with order $|V\Gamma|=11\cdot 3^{d-2}-1$. By Lemma~2.6 in~\cite{LiLuWangTetraSquarefreeOrder}, we have $s\in \{2,3,4,7\}$. The diameter $d$ and the order of $\Gamma$ are in Table~\ref{table diameter and order tetravalent 2AT}. By the list of $2$-arc-transitive tetravalent graphs with at most $512$ vertices~\cite{4valent2at}, the graph $\Gamma$ with order in $\{32,98,296\}$ doesn't exist. So we have $|V\Gamma|=8018=2\cdot 19\cdot 211$. By the list of arc-transitive tetravalent graphs of order $2pq$~\cite{at-order-2pq}, no such graphs exist.
\end{proof}

In the proof of Lemma~\ref{lem 4ndt girth at least 2d is DR}, the case $(c_1',c_2')=(4,2)$ in Case~\ref{case n=2} has two subcases: $P_d(\alpha,\Delta_2)=[1,1]$ or $[2]$. The subcase $P_d(\alpha,\Delta_2)=[1,1]$ can be dealt with in the following paragraph by the adjacency relations in Subsection~\ref{subsec Ps+1alphaDeltaismb}.

Suppose $(c_1',c_2')=(4,2)$ and $P_d(\alpha,\Delta_2)=[1,1]$. Then $a_{1,1}=a_{1,2}=a_{2,1}=0$ and $a_{2,2}=2$. For any vertex $x\in \Gamma_d(\alpha)$, the induced degree $\deg_{[\Gamma_d(\alpha)]}(x)=0$ or $2$. So the induced degree on $[\Gamma_d(\beta)]$ is also $0$ or $2$. The adjacency relations between $\Gamma_s(\alpha)$ and $\Delta_2$ are in Figure~\ref{fig: adjacencyrelationskfourPoneone}. Hence
\begin{eqnarray*}
\Gamma_d(\beta)&=&B_2\cup B_3\cup B_4\cup (\Delta_2)_1',\\
(\Delta_2)_1'&=&\Delta_2\setminus (\Delta_2)_1=\Delta_2(\{2,3\})\cup \Delta_2(\{2,4\})\cup \Delta_2(\{3,4\}).
\end{eqnarray*}
We use notations as in the discussion before Figure~\ref{fig: adjacencyrelationskfourPoneone}. Let $\{j,D,E\}=\{2,3,4\}$. Let $y\in B_j(D_j,E_j)$. Then $y$ has two neighbors in $\Gamma_d(\beta)$ with one in $\Delta_{D_j}(D_j,E_j)$ and the other in $\Delta_{E_j}(D_j,E_j)$. So the induced degree $\deg_{[\Gamma_d(\beta)]}(y)=2$. Let $z\in B_j(1_j,D_j)$. Then $z$ has only one neighbor in $\Gamma_d(\beta)$ (more precisely, the neighbor is in $\Delta_{D_j}(1_j,D_j)$. Let $w\in B_j(1_j,E_j)$. Then $w$ has only one neighbor in $\Gamma_d(\beta)$ (more precisely, the neighbor is in $\Delta_{E_j}(1_j,E_j)$). So the induced degree is $\deg_{[\Gamma_d(\beta)]}(z)=\deg_{[\Gamma_d(\beta)]}(w)=1$. This is a contradiction.

Tetravalent distance-regular graphs were classified in~\cite{4dr}, and these graphs are either distance-transitive~\cite{Smith Distance-transitive graphs of valency four,Smith On bipartite tetravalent graphs,Smith On tetravalent graphs} or the incidence graph of generalized quadrangle $GQ(3,3)$ and the flag graph of generalized hexagon $GH(2,2)$ which are not vertex transitive. So a vertex-transitive and distance-regular graph with valency four must be distance-transitive.

\begin{corollary}\label{cor 4ndt is DT}
Let $\Gamma$ be a tetravalent graph with diameter $d\geqslant 3$ and girth $g\geqslant 2d$. Let $s=d-1$ and let $G\leqslant \Aut(\Gamma)$ such that $\Gamma$ is $(G,s)$-distance-transitive. Then $\Gamma$ is distance-transitive.
\end{corollary}

%
%
%
%

\begin{theorem}\label{thm 4ndt girth at most 2d-1}
Let $\Gamma$ be a tetravalent graph with diameter $d\geqslant 3$ and girth $g$. Let $s=d-1$ and let $G\leqslant \Aut(\Gamma)$ such that $\Gamma$ is $(G,s)$-distance-transitive but not $G$-distance-transitive. Then $g\leqslant 2d-1$.
\end{theorem}


\begin{proof}
We prove this by contrary. Suppose $g\geqslant 2d$. By Lemma~\ref{lem 4ndt girth at least 2d is DR}, the graph $\Gamma$ is distance-regular with intersection array
\begin{equation*}
\iota(\Gamma)=\left\{\begin{array}{cccccc}
               *     & 1 & 1 & \cdots & 1 & c_d\\
               0 & 0 & 0 & \cdots & 0 & a_d\\
               4 & 3 & 3 & \cdots & 3 & *
               \end{array}\right\}.
\end{equation*}
By the classification of distance-regular graphs of valency four~\cite{4dr}, the order of $\Gamma$ is in $\{26,35,728\}$. Magma code of tetravalent $2$-arc-transitive graph of order $26$ or $35$ can be found on Primo\v{z} Poto\v{c}nik's homepage~\cite{PrimozPotocnikHomepage}.

The graph of order $728$ is denoted by $\Gamma_{728}$. The graph $\Gamma_{728}$ is $7$-arc-transitive, and it is the incidence graph of the known generalized hexagon $GH(3,3)$ with full automorphism group $\G_2(3).2$ and stabilizer of order $11664$. The almost simple group $\G_2(3).2$ with socle $\G_2(3)$ is in the database of almost simple groups of Magma. The group $\G_2(3).2$ has only one conjugate class of subgroups of order $11664$. The graph $\Gamma_{728}$ is isomorphic to the orbital graph with arc-transitive group $\G_2(3).2$ and an orbit of length four under the action of the stabilizer, and so it can be constructed by Magma.

By Magma~\cite{Magma}, we can check that the graph $\Gamma$ of order in $\{26,35,728\}$ doesn't possess an automorphism group $G\leqslant \Aut(\Gamma)$ such that $\Gamma$ is $(G,s)$-distance-transitive but not $G$-distance-transitive.
%
\end{proof}


The following is another proof of Theorem~\ref{thm 4ndt girth at most 2d-1}.

\begin{proof}
We prove this by contrary. Suppose $g\geqslant 2d$. By Lemma~\ref{lem 4ndt girth at least 2d is DR}, the graph $\Gamma$ is $(G,s)$-arc-transitive, and it is distance-regular with intersection array
\begin{equation*}
\iota(\Gamma)=\left\{\begin{array}{cccccc}
               *     & 1 & 1 & \cdots & 1 & c_d\\
               0 & 0 & 0 & \cdots & 0 & a_d\\
               4 & 3 & 3 & \cdots & 3 & *
               \end{array}\right\}.
\end{equation*}
Then
\begin{eqnarray*}
|V\Gamma|&=&\sum\limits_{i=0}^d|\Gamma_i(\alpha)|\\
&=&1+4+4\cdot 3+\cdots + 4\cdot 3^{d-2}+\frac{3}{c_d}\cdot 4\cdot 3^{d-2}\\
&=&\left(6+\frac{12}{c_d}\right)\cdot 3^{d-2}-1\\
&=&\left\{
\begin{array}{ll}
18\cdot 3^{d-2}-1, & \text{if $c_d=1$};\\
12\cdot 3^{d-2}-1, & \text{if $c_d=2$};\\
10\cdot 3^{d-2}-1, & \text{if $c_d=3$};\\
9\cdot 3^{d-2}-1, & \text{if $c_d=4$}.\\
\end{array}\right.
\end{eqnarray*}
Since $\Gamma$ is not $G$-distance-transitive, it is not $(G,d)$-arc-transitive. So $\Gamma$ is $(G,s)$-transitive. By Lemma~2.6 in~\cite{LiLuWangTetraSquarefreeOrder}, we have $s\in \{2,3,4,7\}$ and $d=s+1\in \{3,4,5,8\}$. The possible order of $\Gamma$ are in Table~\ref{table order of dr tetravalent 2AT}. By the classification of tetravalent distance-regular graphs~\cite{4dr}, the order $|V\Gamma|\leqslant 728$. Then by Table~\ref{table order of dr tetravalent 2AT}, we have $d\neq 8$, i.e. $d\in \{3,4,5\}$. Note that $\Gamma$ is $2$-arc-transitive. By the list of $2$-arc-transitive tetravalent graphs with at most $512$ vertices~\cite[Table~3 on pages 1334-1335]{4valent2at}, the $(G,s)$-transitive graph $\Gamma$ with $d\in \{3,4,5\}$ and order in Table~\ref{table order of dr tetravalent 2AT} doesn't exist.
\end{proof}

\begin{table}[!ht]
\centering
    \begin{tabular}{|c|c|c|c|c|}
        \hline
         $|V\Gamma|$ & $d=3$ & $d=4$ & $d=5$ & $d=8$ \\
        \hline
        $c_d=1$ & $53$ & $161$ & $485$ & $13121$ \\
        \hline
        $c_d=2$ & $35$ & $107$ & $323$ & $8747$ \\
        \hline
        $c_d=3$ & $29$ & $89$ & $269$ & $7289$ \\
        \hline
        $c_d=4$ & $26$ & $80$ & $242$ & $6560$ \\
        \hline
    \end{tabular}
    \caption{The possible order of the graph $\Gamma$ with girth $g\geqslant 2d$.}\label{table order of dr tetravalent 2AT}
\end{table}

\end{document}